\documentclass[11pt]{amsart}
\author{Kevin Hughes}
\title{Problems and Results related to Waring's problem: Maximal functions and ergodic averages}
\usepackage{amssymb}
\usepackage{amsmath,amsthm}
\usepackage{verbatim}
\usepackage[colorlinks = true, citecolor = cyan]{hyperref}
\newcommand{\dimension}{d}
\newcommand{\degree}{k}

\newcommand{\C}{\mathbb{C} }
\newcommand{\R}{\mathbb{R} }

\newcommand{\euclideanspace}{\R^{\dimension}}

\newcommand{\Z}{\mathbb{Z} }
\newcommand{\lattice}{\Z^{\dimension}}

\newcommand{\T}{\mathbb{T} }
\newcommand{\torus}{\mathbb{T}^\dimension}
\newcommand{\Q}{\mathbb{Q} }
\newcommand{\N}{\mathbb{N} }

\newcommand{\interval}{I}
\newcommand{\Zmod}[1]{\Z(#1)}
\newcommand{\unitsmod}[1]{\Z(#1)^\times}
\newcommand{\discretecube}[1]{C(#1)}
\newcommand{\unitinterval}{[0,1]}
\newcommand{\unitcube}{[-1/2,1/2]^\dimension}
\newcommand{\norm}[1]{\left\Arrowvert #1 \right\Arrowvert}

\newcommand{\Lpnorm}[3]{\left\| #3 \right\|_{L^{#1}(#2)}}

\newcommand{\lpnorm}[2]{\left\| #2 \right\|_{\ell^{#1}(\Z^\dimension)}}

\newcommand{\disup}[3]{\sup_{#1 \leq #2 < 2#1} #3}

\newcommand{\inparentheses}[1]{\left( #1 \right)}
\newcommand{\inbrackets}[1]{\left[ #1 \right]}
\newcommand{\inbraces}[1]{\left\{ #1 \right\}}
\newcommand{\sizeof}{\#}
\newcommand{\volumeof}{Vol}
\newcommand{\absolutevalueof}[1]{\left\lvert #1 \right\rvert}
\newcommand{\innerproductof}[2]{\langle #1,#2 \rangle}
\newcommand{\conjugate}[1]{\overline{#1}} 

\newcommand{\eof}[1]{e \left({#1} \right)} 

\newcommand{\arithmeticFT}[1]{\widehat{#1}}
\newcommand{\latticeFT}[1]{\widehat{#1}}
\newcommand{\torusFT}[1]{\widehat{#1}}
\newcommand{\contFT}[1]{\widetilde{#1}}

\newcommand{\spacepoint}{x}
\newcommand{\freqpoint}{\xi} 
\newcommand{\latticepoint}{m}
\newcommand{\toruspoint}{\xi}

\newcommand{\fxn}{f}
\newcommand{\measure}{\mu}

\newcommand{\inverse}[1]{\frac{1}{#1}}
\newcommand{\gradient}{\nabla}

\newcommand{\indicator}[1]{\mathbf{1}_{#1}}
\newcommand{\bumpfxn}{\varphi}

\newcommand{\convolvedwith}{*}

\newcommand{\uptoconstants}{\eqsim} 

\newcommand{\floor}[1]{\lfloor #1 \rfloor}

\newcommand{\Fareyfraction}{\unit/\modulus}

\newcommand{\WaringFareylevel}{R^{k-1}}
\newcommand{\WaringMajorarclevel}{R}

\newcommand{\acceptableradii}{\mathcal{R}_{\degree,\dimension}} 
\newcommand{\modulus}{q}
\newcommand{\bigmodulus}{Q}
\newcommand{\unit}{a}

\newcommand{\twistedGausssum}[1]{G(\unit,\modulus,#1)}

\newcommand{\numberoflatticepoints}{N_{\degree,\dimension}}
\newcommand{\numberoflatticepointsonspheres}{N_{2,\dimension}}
\newcommand{\numberoflatticepointsonfourspheres}{N_{2,4}}

\newcommand{\hypothesis}[1]{H_{\degree}\left(#1 \right)}
\newcommand{\powerloss}{\theta} 
\newcommand{\powersavings}{\nu} 
\newcommand{\dyadicsavings}{\nu} 
\newcommand{\nonzeroasymptoticshold}{\widetilde{G}_1(\degree)} 

\newcommand{\arc}{I(\Fareyfraction)}

\newcommand{\majorarcs}{\mathfrak{M}}
\newcommand{\minorarcs}{\mathfrak{m}}

\newcommand{\radius}{r}

\newcommand{\Fareysequence}[1]{\mathfrak{F}_{#1}}

\newcommand{\arithmeticsphere}[1]{S^{\degree,\dimension}_{#1}} 
\newcommand{\arithmetichigherordersphere}{S^{\degree,\dimension}_{\radius}} 
\newcommand{\charfxnofarsphere}{{\arithmetichigherordersphere}} %

\newcommand{\multerror}{{E}_{\radius}}
\newcommand{\error}{{E}_{\radius}}

\newcommand{\HLmultiplier}{c}
\newcommand{\HLop}{C_{\radius}}
\newcommand{\maxHLop}{C_*}
\newcommand{\dyadicmaxHLop}{C_{\dyadicradius}}
\newcommand{\dyadicradius}{R}

\newcommand{\higherordersphere}{\mathcal{S}^{\degree,\dimension}_{\radius}} 
\newcommand{\sphere}[1]{\mathcal{S}_{#1}} 
\newcommand{\conthigherordersphere}{\mathcal{S}_{\radius}^{\degree,\dimension}} 
\newcommand{\unithigherordersphere}{\mathcal{S}^{\degree,\dimension}_1} 

\newcommand{\surface}{\mathcal{S}}
\newcommand{\spheremeasure}{\sigma}
\newcommand{\higherorderspheremeasure}{\sigma}
\newcommand{\surfacemeasure}{\sigma}

\newcommand{\GLform}{d\sigma}

\newcommand{\amplitudefxn}{a}

\newcommand{\normalat}[1]{v_{#1}}
\newcommand{\tangentspaceat}[1]{T_{#1}}

\newcommand{\contoperator}{N}
\newcommand{\contmultiplier}{n}

\newcommand{\sphereaverage}[2]{\mathcal{A}_{#1}#2} 
\newcommand{\spheremaximal}[1]{\mathcal{A}^* #1} 
\newcommand{\higherorderspheremultiplier}{\arithmeticFT{A_{\radius}}}
\newcommand{\arithmetichigherordersphereaverage}{A_{\radius}} 
\newcommand{\arithmetichigherorderspheremaximal}[1]{A^* #1} 

\newcommand{\averagefxn}{A_{\radius}\fxn} 
\newcommand{\avgop}{A_{\radius}} 
\newcommand{\dyadicmaxfxn}{A_{\dyadicradius}\fxn} 
\newcommand{\maxfxn}{A_*\fxn}
\newcommand{\maxop}{A_*}
\newcommand{\dyadicmaxop}{A_\dyadicradius}

\newcommand{\contaveragefxn}{\mathcal{A}_{\radius}\fxn} 
\newcommand{\contavgop}{\mathcal{A}_{\radius}} 
\newcommand{\contmaxfxn}{\mathcal{A}_*\fxn}
\newcommand{\contmaxop}{\mathcal{A}_*}

\newcommand{\measurespace}{X}
\newcommand{\measurespacepoint}{x}
\newcommand{\measurepreservingtransform}{T}

\newcommand{\mps}{(\measurespace, \measure, \measurepreservingtransform)} 

\newcommand{\arithmeticergodicsphericalaverage}{\mathfrak{A}_{\radius}}
\newcommand{\arithmeticergodicsphericalmaximal}{\mathfrak{A}^{*}}
\newcommand{\truncatedarithmeticergodicsphericalmaximal}[1]{\mathfrak{A}^{*}_{#1}}

\newcommand{\truncatedarithmetichigherorderspheremaximal}[1]{A^{*}_{#1}}

\newcommand{\spectralmeasure}{\nu}
\newcommand{\spectralpoint}{\eta}

\newcommand{\transferfxn}{F}
\newcommand{\mfxn}{f}

\newcommand{\transferHLoponarc}{{C}_{\radius}^{\Fareyfraction}}

\newcommand{\transfermaxHLoponarc}{C_{*}^{\Fareyfraction}}

\theoremstyle{plain}

\theoremstyle{definition}

\theoremstyle{remark}
\newtheorem{rem}{Remark}[section]

\theoremstyle{plain}

\newtheorem{prop}{Proposition}[section]
\newtheorem{lemma}{Lemma}[section]
\newtheorem{thm}{Theorem}
\newtheorem*{fact}{Fact}
\newtheorem{cor}{Corollary}[section]

\newtheorem{conjecture}{Conjecture}

\newtheorem*{Huabound}{Hua's bound for Gauss sums}

\newtheorem*{Hypothesis}{Hypothesis $\hypothesis{\powerloss}$}
\newtheorem*{Wooley}{Wooley}

\newtheorem*{MSW}{Magyar--Stein--Wainger}
\newtheorem*{MSW_transference}{Magyar--Stein--Wainger transference lemma}
\newtheorem*{Magyar}{Magyar}
\newtheorem*{Ionescu}{Ionescu}

\newtheorem*{dyadic_major_arc_approximation}{Dyadic Major Arc Approximation Lemma}
\newtheorem*{approximationlemma}{The Approximation Formula}

\newtheorem*{Stein}{Stein}

\newtheorem*{BNW}{Bruna--Nagel--Wainger}
\newtheorem*{RubiodeFrancia}{Rubio de Francia maximal theorem}

\newtheorem*{mpsmaximaltheorem}{Maximal Theorem for Measure Preserving Systems}
\newtheorem*{Ltwoergodictheorem}{$L^2$ Ergodic Theorem} 

\newtheorem*{pointwiseergodictheorem}{Pointwise Ergodic Theorem}

\newtheorem*{spectraltheorem}{The Spectral Theorem}
\newtheorem*{oscillation_inequality_mps}{Oscillation Inequality for Measure Preserving Systems}
\newtheorem*{oscillation_inequality_lattice}{Transfer Oscillation Inequality}

\begin{document}
\maketitle 
%

%
\begin{abstract}
We study the arithmetic analogue of maximal functions on diagonal hypersurfaces. 
This paper is a natural step following the papers of \cite{Magyar_dyadic},  \cite{Magyar_ergodic} and \cite{MSW}. 
We combine more precise knowledge of oscillatory integrals and exponential sums to generalize the asymptotic formula in Waring's problem to an approximation formula for the fourier transform of the solution set of lattice points on hypersurfaces arising in Waring's problem and apply this result to arithmetic maximal functions and ergodic averages. 
In sufficiently large dimensions, the approximation formula, $\ell^2$-maximal theorems and ergodic theorems were previously known. 
Our contribution is in reducing the dimensional constraint in the approximation formula using recent bounds of Wooley, and improving the range of $\ell^p$ spaces in the maximal and ergodic theorems. We also conjecture the expected range of spaces.
\end{abstract}

\section{Introduction} 
In this paper we generalize the asymptotic formula in Waring's problem to an approximation formula for the Fourier transform of the solution set of lattice points on a certain class of hypersurfaces. Next we apply this result to arithmetic maximal functions and ergodic averages. The approximation formula was previously known in sufficiently large dimensions while the maximal and ergodic theorems were previously known only for $\ell^2$ -- see \cite{Magyar_ergodic}. 
Using recent bounds of Wooley, our contribution is an improved error estimate in the approximation formula  and an improved range of $\ell^p$ spaces in the maximal and ergodic theorems. We also conjecture the expected range of spaces.
In related papers we investigate applications to Szemer\'edi theorems as in \cite{Magyar_distance}, discrepancy theory as in \cite{Magyar_discrepancy} and restriction theory as in \cite{Hu_Li1}. 

\subsection{Arithmetic maximal functions}
Fix the degree $\degree \geq 2$ and dimension $\dimension$, positive integers. We define the \emph{arithmetic $\degree$-sphere of radius $\radius$ in $\dimension$ dimensions} as 
\[
\arithmetichigherordersphere := \{\latticepoint \in \Z^\dimension : \sum_{i=1}^{\dimension} |\latticepoint_i|^\degree=r^\degree\} 
. 
\]
$\arithmetichigherordersphere$, the arithmetic $\degree$-sphere of radius $\radius$, contains $\numberoflatticepoints(\radius) = \# \arithmetichigherordersphere$ lattice points. 
$\arithmetichigherordersphere$ is possibly non-empty only when $\radius^\degree \in \N$; we denote the set of positive radii $\radius$ such that $\arithmetichigherordersphere \not= \emptyset$ by $\acceptableradii$. 
For a function $\fxn : \lattice \to \C$ and $\radius \in \acceptableradii$, we introduce the \emph{$\degree$-spherical averages}, \emph{dyadic $\degree$-spherical maximal function} and \emph{(full) $\degree$-spherical maximal function} respectively,  

\begin{align*}
\averagefxn(x) &= \inverse{\numberoflatticepoints(\radius)} \sum_{y \in \arithmetichigherordersphere} \fxn(x-y) , \\ 
\dyadicmaxfxn &= \disup{\dyadicradius}{\radius} \absolutevalueof{\averagefxn} , \\ 
\maxfxn &= \sup_{\radius \in \acceptableradii} \absolutevalueof{\averagefxn} 
.
\end{align*}

\begin{rem}
Throughout, all averages will be restricted to $\radius \in \acceptableradii$; that is, only $\degree$-spheres with lattice points on them; in particular, the dyadic supremum above is restricted to $\radius \in \acceptableradii$. 
\end{rem}

These maximal functions are the arithmetic analogues of continuous maximal functions over $\degree$-spheres in Euclidean space. In the continuous setting, maximal functions associated to compact convex hypersurfaces are bounded on a range of $L^p(\R^\dimension)$ spaces depending on the geometry of the hypersurface. Therefore, it is natural to ask: \emph{when is $\maxop$ bounded on $\ell^p(\Z^\dimension)$?} 
For sufficiently large $\radius \in \acceptableradii$, $\numberoflatticepoints(\radius) \uptoconstants \radius^{\dimension-\degree}$ when $\dimension$ is sufficiently large with respect to $\degree$. 
Testing the maximal operator on the Dirac delta function ($\delta(\latticepoint)$ is 1 if $\latticepoint = 0$ and 0 otherwise), we expect that the maximal operator is bounded on $\ell^p$ for $p>\frac{\dimension}{\dimension-\degree}$. Let $\nonzeroasymptoticshold$ be the smallest dimension such that $\numberoflatticepoints(\radius) \uptoconstants \radius^{\dimension-\degree}$. 
By the Hardy--Littlewood circle method and Jacobi's 4-squares formula, we know that $\nonzeroasymptoticshold = 5$ for $\degree=2$, see Theorem 4.1 on p. 20 of \cite{Davenport} and Theorem 3.6 on p. 304 of \cite{SteinComplexAnalysis} respectively. The asymptotics for $\numberoflatticepoints(\radius)$ when $\degree \geq 3$; in particular, the value of $\nonzeroasymptoticshold$ is an open problem in number theory. Recent work of Wooley, in particular Theorem 1.4 on page 4 of \cite{Wooley_efficient_congruencing1}, shows that $\nonzeroasymptoticshold \leq 2\degree^2 + 2\degree - 3$ for $\degree \geq 3$. We will always assume that the dimension $\dimension \geq 2\degree^2 + 2\degree - 3$ so that this holds. 

\subsection{Previous results and conjectures}
In \cite{Magyar_dyadic}, Magyar initiated the study arithmetic $\degree$-spherical maximal functions and proved that the dyadic maximal operator is bounded uniformly in $\dyadicradius$ for a range of $\ell^p(\Z^\dimension)$ spaces depending on the degree and dimension. 
\begin{Magyar}
If $\degree = 2$, then $\dyadicmaxop$ is bounded on $\ell^{p}(\Z^\dimension)$ for $p> \frac{\dimension}{\dimension-{2}}$ and $\dimension \geq 5$.
If $\degree \geq 3$, then $\dyadicmaxop$ is bounded on $\ell^{p}(\Z^\dimension)$ for $p> \frac{\dimension}{\dimension-{\degree 2^{\degree}}}$ and $\dimension > \degree 2^{\degree+1}$.
\end{Magyar}
For continuous maximal functions, the boundedness of the dyadic maximal operator is equivalent to the full maximal function by the use of Littlewood--Paley theory. However, this argument fails in the arithmetic setting, and a new idea is needed to understand the full maximal function. Building on Magyar's work, \cite{MSW} studied the full maximal function for degree $\degree = 2$ proving the arithmetic analogue of Stein's spherical maximal theorem -- see \cite{Steinsphericalmaximalfunction}. 
\begin{MSW}
Let $\degree=2$ and $\dimension \geq 5$, then $\maxop$ is bounded on $\ell^p(\Z^\dimension)$ for $p>\frac{\dimension}{\dimension-2}$.
\end{MSW}
\begin{rem}
If $\dimension \geq 5$, then $\numberoflatticepointsonspheres(\radius) \uptoconstants \radius^{\dimension-2}$. Testing $\maxop$ on the delta function, one deduces that $\maxop$ is unbounded on $\ell^p(\lattice)$ for $p\leq \frac{\dimension}{\dimension-2}$. For dimensions $\dimension \leq 4$ the maximal function is only bounded on $\ell^\infty(\Z^\dimension)$. This is because in $\Z^4$, there are precisely 24 lattice points on a sphere of radius $2^j$; that is, $\numberoflatticepointsonfourspheres(2^j) = 24$ for all $j \in \N$. 
\end{rem}

Subsequently, \cite{Ionescu} improved the Magyar--Stein--Wainger result by proving the restricted weak-type result at the endpoint $p=\frac{\dimension}{\dimension-2}$. 
\begin{Ionescu}
Let $\degree=2$ and $\dimension \geq 5$, then $\maxop$ is bounded from $\ell^{\frac{\dimension}{\dimension-2},1}_{rest}(\Z^\dimension)$ to $\ell^{\frac{\dimension}{\dimension-2},\infty}(\Z^\dimension)$.
\end{Ionescu}
This result is analogous to the Bourgain's restricted weak-type result for the continuous spherical maximal function in 3 or more dimensions -- see \cite{Bourgain}.

\cite{Magyar_ergodic} extended the results of Magyar--Stein--Wainger to positive definite, nondegenerate, homogeneous integral forms to prove boundedness of the corresponding maximal operator on $\ell^2(\Z^\dimension)$ and pointwise convergence of their ergodic averages when $\dimension>(\degree-1)2^\degree$; this includes the family of $\degree$-spheres considered here. 
%
%
Based on these examples and results, we conjecture when $\maxop$ is bounded on $\ell^p(\lattice)$ for $\degree \geq 3$.
\begin{conjecture}\label{conjecture:higher_degree_strong_type}
If $\dimension \geq \nonzeroasymptoticshold$, then $\maxop$ is bounded on $\ell^p(\Z^\dimension)$ for $p>\frac{\dimension}{\dimension-\degree}$.
\end{conjecture}

\begin{conjecture}\label{conjecture:higher_degree_weak_type}
If $\dimension \geq \nonzeroasymptoticshold$, then $\maxop$ is bounded from $\ell^{\frac{\dimension}{\dimension-\degree},1}(\Z^\dimension)$ to $\ell^{\frac{\dimension}{\dimension-\degree},\infty}(\Z^\dimension)$.
\end{conjecture}

\begin{rem}
By interpolation with the trivial bound for $\maxop$ on $\ell^\infty(\Z^\dimension)$, Conjecture 2 implies Conjecture~1.
\end{rem}

\subsection{Summary of results}
An important and novel ingredient in Magyar--Stein--Wainger's result is their approximation formula. We extend this to higher degrees as in \cite{Magyar_ergodic}, but now take advantage of refined knowledge for exponential sums and oscillatory integrals related to $\degree$-spheres.
%
\begin{approximationlemma}
Let $\charfxnofarsphere(\latticepoint)$ be the characteristic function of $\arithmetichigherordersphere$ on $\lattice$. If $\degree \geq 3$ and $\dimension>2\degree^2(\degree-1)$, then for $\toruspoint \in \torus$
\[
\arithmeticFT{\charfxnofarsphere}(\xi) 
= \sum_{\modulus=1}^{\infty} \sum_{\unit \in \unitsmod{\modulus}} \eof{ \frac{\unit \radius^\degree}{\modulus} } \sum_{\latticepoint \in \lattice} \twistedGausssum{\latticepoint} \Psi(\modulus \xi - \latticepoint) \contFT{d\surfacemeasure_\radius}(\modulus \xi - \latticepoint) + \arithmeticFT{\error}(\xi)
\]
and 
\[
\lpnorm{2}{\disup{\dyadicradius}{\radius}{\error}}
\lesssim \dyadicradius^{\dimension-\degree-\powersavings} 
\]
for some $\powersavings(\degree,\dimension)>0$. 
\end{approximationlemma}
Here and throughout the paper, 
$\twistedGausssum{\latticepoint} := \modulus^{-\dimension} \sum_{b \in \Zmod{\modulus}} \eof{\frac{\unit b^\degree + b \cdot \latticepoint}{\modulus}}$ 
is the Gauss sum of degree $\degree$ while $\surfacemeasure_\radius$ is the Gelfand--Leray measure on the continuous $\degree$-sphere 
$\conthigherordersphere := \inbraces{x \in \R^\dimension : \sum_{i=1}^{\dimension} \absolutevalueof{x_i}^\degree = \radius^\degree}$ 
with $\contFT{d\surfacemeasure_\radius}$ as its $\R^\dimension$-variable Fourier transform. $\Psi$ is a smooth function supported in $\inbrackets{-1/4, 1/4}$ and 1 in $\inbrackets{-1/8, 1/8}$. The approximation lemma says that we can approximate the $\lattice$-variable Fourier transform of the arithmetic surface measure of a $\degree$-sphere as a weighted sum of pieces of a localized $\R^\dimension$-variable Fourier transform of the continuous $\degree$-sphere with $\error$, an error term that has a power saving in the radius.

\begin{thm}\label{thm:Wooley_max_fxn_bound}
If $\degree \geq 3$ and $\dimension>2\degree^2(\degree-1)$, then $\maxop$ is bounded on $\ell^p(\Z^\dimension)$ for $p>\frac{\dimension}{\dimension - \degree^2(\degree-1)}$.
\end{thm}
We refine this to a restricted weak-type endpoint result in a subsequent paper. 
%
%
In \cite{Magyar_ergodic}, Magyar also investigated related ergodic theorems. We improve his results for $\degree$-spheres. 
Suppose that we have a probability space $\measurespace$ with measure $\measure$ and a strongly ergodic (commuting) family $\measurepreservingtransform = (\measurepreservingtransform_1, \dots, \measurepreservingtransform_\dimension)$ of invertible measure preserving transformations. We use these transformations to define actions on $\measurespace$ and functions. For a function $\mfxn:\measurespace \to \C$, define the \emph{$\degree$-spherical average of radius} $\radius$ as 
\[
\arithmeticergodicsphericalaverage \mfxn(x) 
:= \inverse{N_{\dimension,\degree}(\radius)} \sum_{n \in \arithmetichigherordersphere} \mfxn(\measurepreservingtransform^n x)
.
\]
Using Theorem~\ref{thm:Wooley_max_fxn_bound} and proving an oscillation inequality similar to equation (6.6) in \cite{Magyar_ergodic}, we prove a pointwise ergodic theorem.
\begin{thm}
If $\degree \geq 3$, $\dimension > 2\degree^2(\degree-1)$ and $f \in L^p(\measurespace,\measure)$ for some $p>\frac{\dimension}{\dimension - \degree^2(\degree-1)}$, then 
\[
\lim_{\radius \to \infty} \arithmeticergodicsphericalaverage f(x) = \int_\measurespace f \, d\measure
\]
for $\measure$-a.e. $x\in \measurespace$.
\end{thm}
\begin{rem}
The ranges $\dimension > 2\degree^2(\degree-1)$ and $p>\frac{\dimension}{\dimension - \degree^2(\degree-1)}$ in Theorem~\ref{thm:Wooley_max_fxn_bound} are chosen for aesthetic reasons. Theorem~\ref{thm:Vinogradov_max_thm} gives a more flexible version of Theorem~\ref{thm:Wooley_max_fxn_bound} which depends crucially on supremum bounds for a certain class of exponential sums. We phrase our bounds using a hypothesis on these exponential sums that is based on works in Waring's problem. Then Theorem~\ref{thm:Wooley_max_fxn_bound} is deduced by using Wooley's sup bound, \cite{Wooley_efficient_congruencing1}, Theorem 1.5, p. 5. There are immediate improvements for $\degree \geq 4$. We discuss the best currently known results and conjectural limitations of our method in Section~\ref{section:general_thm}. 
\end{rem}
%

\subsection{Structure of the paper}
In section~\ref{section:notations}, we briefly mention notation used througout the paper. We hope that there will be a mix of readers from harmonic analysis, ergodic theory and number theory, so we try to make the exposition as self-contained as possible. 
In section~\ref{section:cont_max_fxns}, we discuss continuous maximal functions over hypersurfaces and derive bounds for continuous $\degree$-spherical maximal functions. 
In section~\ref{section:MSW_machinery}, we state some of the important machinery in \cite{MSW}; in particular we will need the Magyar--Stein--Wainger transference principle in order to exploit the machinery in section~\ref{section:cont_max_fxns}. 
In section~\ref{section:general_thm}, we introduce a hypothesis for exponential sums that allows us to generalize The Approximation Formula and Theorem~\ref{thm:Wooley_max_fxn_bound}. We state Wooley's recent bounds in \cite{Wooley_efficient_congruencing1}. Theorem~\ref{thm:Wooley_max_fxn_bound} will then be an immediate application of Wooley's bounds and Theorem~\ref{thm:Vinogradov_max_thm}. 
In section~\ref{section:dyadic}, we study the dyadic maximal operators. 
In section~\ref{section:approximation_formula}, we prove The Approximation Formula of section~\ref{section:general_thm}. 
In section~\ref{section:maximal_theorem_proof}, we combine the analysis in sections \ref{section:dyadic} and \ref{section:approximation_formula} to prove Theorem~\ref{thm:Vinogradov_max_thm}. 
In section~\ref{section:ergodic}, we use Theorem~\ref{thm:Vinogradov_max_thm} to a mean $L^2$ ergodic theorem and our pointwise ergodic theorem for $\degree$-spherical averages. 

\section{Notations}\label{section:notations}
Before discussing the machinery in proving Theorem~\ref{thm:Wooley_max_fxn_bound}, we introduce some notation. Our notation will be a mix of notations from analytic number theory and harmonic analysis. Most of our notation is standard, but there are a few differences based on aesthetics. 

\begin{itemize}
\item The torus $\torus$ may be identified with any box in $\euclideanspace$ of sidelengths 1, for instance $[0,1]^\dimension$ or $[-1/2,1/2]^\dimension$.
\item $\eof{t}$ will denote the character $e^{2 \pi i t}$ for $t \in \R, \Z$ or $\T$.
\item We use the non-standard notation of $\Zmod{\modulus} = \Z/ \modulus \Z$ which we identify with the set $\inbraces{1, \cdots, \modulus}$ and $\unitsmod{\modulus}$ is the group of units in $\Zmod{\modulus}$. 
\item For $\degree \in \N$, define $\absolutevalueof{x}^{\degree} := \sum_{i=1}^\dimension \absolutevalueof{x_i}^{\degree}$ if $x \in \euclideanspace$, $\lattice$ or $\torus$ and use the dot product $\cdot$. Furthermore, we abuse notation by writing $\absolutevalueof{b}^\degree$ to mean $\sum_{i=1}^{\dimension} b_i^\degree$ for $b \in \Zmod{\modulus}^{\dimension}$ and the dot product notation $b \cdot m$ to mean $\sum_{i=1}^{\dimension} b_i m_i$ for $b, m \in \Zmod{\modulus}$. 
\item For two functions $f, g$, $f \lesssim g$ if $\absolutevalueof{f(x)} \leq C \absolutevalueof{g(x)}$ for some constant $C>0$. $f(x) \lessapprox x^n g(x)$ if $\absolutevalueof{f(x)} \leq C_\epsilon |x|^{n+\epsilon} |g(x)|$ for each $\epsilon >0$ with $C_{\epsilon}$ depending on $\epsilon$. $f$ and $g$ are comparable $f \uptoconstants g$ if $f \lesssim g$ and $g \lesssim f$. Finally, we may use $f \ll g$ if $|f|$ is much smaller than $|g|$. 
All constants throughout the paper may depend on dimension $\dimension$ and degree $\degree$. 
\item 
If $f: \euclideanspace \to \C$, then we define its Fourier transform by $\contFT{\fxn}(\xi) := \int_{\euclideanspace} f(x) e(x \cdot \xi) dx$ for $\xi \in \euclideanspace$. 
If $f: \torus \to \C$, then we define its Fourier transform by $\torusFT{\fxn}(\latticepoint) := \int_{\torus} f(x) e(-\latticepoint \cdot x) dx$ for $\latticepoint \in \lattice$. 
If $f:\lattice \to \C$, then we define its Fourier transform by $\latticeFT{\fxn}(\xi) := \sum_{\latticepoint \in \lattice} f(\latticepoint) e(n \cdot \xi)$ for $\xi \in \torus$. 
\end{itemize}

We introduce several related convolution operators such as the averaging operators $\contavgop$, $\avgop$ and $\arithmeticergodicsphericalaverage$ defined on $\euclideanspace$, $\lattice$ and a measure space $\measurespace$, respectively. The averages are intimately connected to one another and we will distinguish them by using mathcal font for an operator on $\euclideanspace$, normal font for operators on $\lattice$ and mathfrak font for operators on $\measurespace$.  

\section{Estimates for continuous maximal functions over hypersurfaces}\label{section:cont_max_fxns}
In this section we discuss the continuous analogues of our arithmetic $\degree$-spherical maximal functions. There is a wide literature on continuous maximal functions over hypersurfaces. We discuss two results, one due to Bruna--Nagel--Wainger and another due to Rubio de Francia. We then apply them to deduce $L^p(\euclideanspace)$ estimates for continuous $\degree$-spherical maximal functions in Proposition~\ref{prop:cont_max_fxn_bound_for_higher_order_spheres}. In section~\ref{section:approximation_formula} these estimates will be applied later using the Magyar--Stein--Wainger transference principle.

\subsection{Measures on hypersurfaces}
%
There are several natural measures associated with a hypersurface $\surface \subset \euclideanspace$. Before we can state the necessary results, we need to be precise about which surface measure we are using. We suppose that the hypersurface, $\surface = \inbraces{ x \in \euclideanspace : \Phi(x) = 0 }$, is defined by a function $\Phi : \euclideanspace \to \R$ such that $\Phi$ is non-singular, that is, $\gradient \Phi(x) \not= 0$ for $x \in \surface$. 
We use the \emph{Gelfand--Leray measure} which is defined as the unique form $d\surfacemeasure_{\Phi}$ such that
\[
d\Phi \wedge \GLform_\Phi 
= dx_1 \wedge \dots \wedge dx_\dimension
\]
where $d\Phi = \sum_{i=1}^{\dimension} \partial_i \Phi \, dx_i$. 
The Gelfand--Leray measure is equal to the induced Lebesgue measure and the Dirac delta-measure restricted to $\surface$ (for definitions of these measures see \cite{SteinHA}, page 498). The Gelfand--Leray measure is also equal to an appropriately normalized Euclidean surface measure -- see Proposition 2 in \cite{Magyar_ergodic}. However, most important for us is the distributional description of the Gelfand--Leray measure given now. 
\begin{fact}
If $\bumpfxn$ is a Schwartz function on $\R$ and $f$ is a Schwartz function on $\euclideanspace$, then 
\[
\int_{\euclideanspace} f(x) \GLform_{\Phi}(x) 
= \lim_{\epsilon \to 0} \int_{\R} \int_{\euclideanspace} \bumpfxn(\epsilon t) f(x) e(\Phi(x)t) \, dx \, dt 
. 
\]
\end{fact}
In particular, this is true if we take $\bumpfxn(x) = e^{-2\pi \absolutevalueof{x}^\degree}$. 
The proof of this fact follows from a change of variables and the Fourier inversion theorem. For the details of the proof, see Lemma 2 in \cite{Magyar_ergodic}. 
%

\subsection{$L^p$ bounds for maximal functions over hypersurfaces}
Let $\sphere{\radius}$ be the sphere of radius $\radius$ in $\R^\dimension$ with surface area measure $d\spheremeasure_{\radius}$ normalized to have $\spheremeasure_{\radius}(\sphere{\radius}) = \spheremeasure_{1}(\sphere{1})$ for all $\radius > 0$. 
For a continuous function $\fxn$ with compact support, define the spherical average of radius $\radius$ by
\[
\sphereaverage{\radius}{\fxn}(x) 
:= \fxn \convolvedwith \spheremeasure_{\radius} (x)
= \int_{\sphere{1}} \fxn(x-\radius y) \; d\spheremeasure_{1}(y) 
\]
and the spherical maximal function
\[
\spheremaximal{\fxn} 
:= \sup_{\radius>0} \absolutevalueof{\sphereaverage{\radius}{\fxn} }
.
\]
E. Stein was the first to investigate the spherical maximal function, proving boundedness on the sharp range of $L^p$ spaces when $\dimension \geq 3$. 
\begin{Stein}[1976]
In dimensions $\dimension \geq 3$, the spherical maximal function is bounded on $L^p(\R^\dimension)$ when $p>\frac{\dimension}{\dimension-1}$.
\end{Stein}
\begin{rem}
The spherical maximal function is unbounded on $L^p$ for $p \leq \frac{\dimension}{\dimension-1}$. This can be seen by considering the characteristic function of the unit cube, a delta mass at the origin, or the scale invariant version $f(x):= \absolutevalueof{x}^{1-\dimension} \log{ \absolutevalueof{x}^{-1}}$ if $\absolutevalueof{x} \geq 1$ and $f(x)=0$ otherwise.
\end{rem}
%
%

In this section, we are interested in generalizations of Stein's spherical maximal theorem to hypersurfaces. Let $\surface$ be a smooth, convex hypersurface of finite type in $\R^{\dimension}$ for $\dimension \geq 3$ and $\surfacemeasure$ be its Gelfand--Leray measure; we say that a hypersurface is convex if the body it bounds is convex, and a hypersurface is finite type if at each point $x \in \surface$, all tangent lines at $x$ make finite order of contact. Furthermore, let $\amplitudefxn(x)$ be a smooth, positive, compactly supported function on $\surface$ ($\amplitudefxn$ is arbitrary, but all implicit constants below may depend on it). Similarly to the spherical maximal function, define the averages and maximal function respectively by 
\[
\contaveragefxn(x) 
:= \int_\surface \fxn(x-\radius y) \amplitudefxn(x) \, d\surfacemeasure(y)
\]
and 
\[
\contmaxfxn(x) 
:= \sup_{\radius >0} \absolutevalueof{ \contaveragefxn(x) }
.
\]
%
%

Nontrivial $L^p \to L^p$ estimates for the maximal function $\contmaxfxn$ occur when the surface $\surface$ satisfies a curvature condition such as everywhere positive Gaussian curvature in the case of the sphere or the finite type condition. Such a curvature condition is reflected in the decay of the Fourier transform of the surface measure which implies that the associated maximal function is bounded on a range of $L^p$ spaces. 
A surface can fail to have nontrivial $L^p$ bounds when the surface is very flat -- see \cite{CowlingMauceri} for more details. 

Below we choose two exemplary theorems to prove the boundedness of the continuous $\degree$-spherical maximal function. The first theorem, due to Bruna--Nagel--Wainger, relates the decay estimates of the Fourier transform of a surface measure to the curvature of the surface. For $\spacepoint \in \surface$, denote the outward unit normal to $\surface$ at $\spacepoint$ by $\normalat{\spacepoint}$ and $\tangentspaceat{\spacepoint}$ to be the tangent plane at $\spacepoint$. For $\delta>0$, define the $\delta$-ball about $\spacepoint$ as $ \mathcal{B}(\spacepoint,\delta) := \{y \in \surface : dist(y, \tangentspaceat{\spacepoint}) < \delta \}$. 

\begin{BNW} 
If $\surface$ is a smooth, convex hypersurface of finite type in $\R^\dimension$ with surface measure $\surfacemeasure$, then
\[
\contFT{d\surfacemeasure}(t \normalat{\spacepoint} ) 
\lesssim \surfacemeasure({\mathcal{B}(\spacepoint,t^{-1})})
\]
as $t \to \infty$.
\end{BNW}
%

The second theorem, due to Rubio de Francia, relates the decay estimates of the Fourier transform of a surface measure to boundedness of the maximal function on $L^p$ spaces.
\begin{RubiodeFrancia}
Suppose that the dimension $\dimension$ is at least 3. If $\contFT{d\surfacemeasure}(\freqpoint) \lesssim (1+\absolutevalueof{\freqpoint})^{-\gamma}$ for all $\freqpoint$ and some $\gamma > 1/2$, then $\contmaxop$ is a bounded operator on $L^p(\R^\dimension)$ for $p> 1+(2\gamma)^{-1}$.
\end{RubiodeFrancia}
\begin{rem}
Since $\gamma>1/2$, the range of $p$ extends below 2. In particular, $\contmaxop$ is bounded on $L^2(\euclideanspace)$.
\end{rem}
%
\subsection{$L^p$ bounds for maximal functions over $\degree$-spheres}
We now combine the Bruna--Nagel--Wainger and Rubio de Francia theorems to establish $L^p$ bounds for maximal operators associated to $\degree$-spheres. Recall that $\absolutevalueof{x}^\degree := \sum_{i=1}^{\dimension} \absolutevalueof{x_i}^\degree$ and the hypersurface $\conthigherordersphere := \{ x \in \euclideanspace: \absolutevalueof{x}^\degree = \radius^\degree \}$ with $\higherorderspheremeasure_{\radius}$ its normalized Gelfand--Leray measure. $\unithigherordersphere$ is finite type of order $\degree$, and
\begin{equation}\label{BNW_bound}
\higherorderspheremeasure_1(\mathcal{B}(x, \delta)) \lesssim \delta^{\frac{\dimension-1}{\degree}}
\end{equation}
uniformly for $x \in \unithigherordersphere$. \eqref{BNW_bound} is sharp at the poles, e.g. $(\pm1, 0,\dots,0)$. 
Since $\unithigherordersphere$ is a compact surface, we may take $\amplitudefxn \equiv 1$. Thus by the Bruna--Nagel--Wainger theorem, the Fourier decay estimates are 
\begin{equation}\label{kspherical_decay}
\contFT{d\surfacemeasure_{1}}(\freqpoint) 
\lesssim \inparentheses{1+\absolutevalueof{\freqpoint}}^{-\frac{\dimension-1}{\degree}} 
\end{equation} 
uniformly for $\xi \in \R^\dimension$.
Applying the Rubio de Francia theorem, we conclude:
\begin{prop}\label{prop:cont_max_fxn_bound_for_higher_order_spheres}
For $\dimension > \frac{\degree}{2}+1$, $\contmaxop$ is a bounded operator on $L^p(\euclideanspace)$ if $p> 1+\frac{\degree}{2\dimension-2} = \frac{2\dimension-2+\degree}{2(\dimension-1)}$.
\end{prop}
%
%
\begin{rem}
For the continuous maximal functions we know the sharp Fourier decay estimates. However, these do not necessarily imply the sharp maximal function estimates. In particular, the $L^p$ bounds for $\degree$-spheres are not optimal, but they are sufficient for our applications since $\frac{2\dimension-2+\degree}{2\dimension-2}<\frac{\dimension}{\dimension-\degree}$ for $\dimension > \degree$. There have been many results and much progress in this area -- see for instance \cite{IKM}, but the general problem is still open.
\end{rem}
%

\section{Some machinery of Magyar--Stein--Wainger}\label{section:MSW_machinery}
In this section, we review some of the machinery in \cite{MSW}. In particular, we recall the transference principle of Magyar--Stein--Wainger and two inequalities that will be useful later. 

\subsection{The Magyar--Stein--Wainger transference principle}
Given The Approximation Formula, it will be necessary to understand the relationship between multipliers defined on $\torus$ and $\euclideanspace$. Suppose that $\mu$ is a multiplier supported in $\unitcube$, then we can think of $\mu$ as a multiplier on $\euclideanspace$ or $\torus$; denote this as $\mu_{\euclideanspace}$ and $\mu_{\torus}$ respectively where $\mu_{\torus}(\toruspoint):= \sum_{\latticepoint \in \lattice} \mu(\toruspoint-\latticepoint)$ is the periodization of $\mu_{\euclideanspace}$. These have convolution operators $T_{\euclideanspace}$ and $T_{\torus}$ respectively. For $F:\euclideanspace \to \C$,
\[
T_{\euclideanspace}F(\spacepoint) := \int_{\euclideanspace} \mu_{\euclideanspace}(\freqpoint) \contFT{F}(\freqpoint) \eof{-\spacepoint \cdot \freqpoint} \, d\freqpoint
\]
and for $f:\lattice \to \C$,
\[
T_{\lattice} f(\latticepoint) := \int_{\torus} \mu_{\torus}(\toruspoint) \arithmeticFT{f}(\toruspoint) \eof{-\latticepoint \cdot \toruspoint} \, d\toruspoint .
\]
Equivalently, let $K_{\euclideanspace}$ be the kernel of $T_{\euclideanspace}$, 
\[
K_{\euclideanspace}(\spacepoint) = \int_{\euclideanspace} \mu_{\euclideanspace}(\freqpoint) \eof{\spacepoint \cdot \freqpoint} \, d\freqpoint
\]
and $K_{\lattice}$ be the kernel of $T_{\lattice}$, 
\[
K_{\lattice}(\latticepoint) = \int_{\torus} \mu_{\torus}(\toruspoint) \eof{\latticepoint \cdot \freqpoint} \, d\freqpoint .
\]
Then $K_{\euclideanspace}$ is smooth, and $K_{\lattice}$ is $K_{\euclideanspace}\rvert_{\lattice}$, the restriction of $K_{\euclideanspace}$ to the lattice $\lattice$.

We extend these notions to Banach spaces. Let $B_1, B_2$ be two Banach spaces, possibly infinite dimensional, with norms $\norm{\cdot}_1, \norm{\cdot}_2$, and $\mathcal{L}(B_1,B_2)$ is the space of bounded linear tranformations from $B_1$ to $B_2$. Let $\ell^p_{B_i}$ be the space of functions $f:\lattice \to B_i$ such that $\sum_{\latticepoint \in \lattice} \norm{f}_{i}^p < \infty$ and $L^p_{B_i}$ be the space of functions $F:\euclideanspace \to B_i$ such that $\int_{\euclideanspace} \norm{F}_{i}^p < \infty$. For $\modulus \in \N$, suppose that $\mu: [-1/2\modulus,1/2\modulus]^\dimension \to \mathcal{L}(B_1,B_2)$ is a multiplier with convolution operators $T_{\euclideanspace}$ on $\euclideanspace$ and $T_{\lattice}$ on $\lattice$. Extend $\mu$ periodically to $\unitcube$ to define 
$\mu^{\modulus}_{\torus}(\freqpoint) := \sum_{\latticepoint \in \lattice} \mu_{\torus}(\freqpoint-\latticepoint/\modulus)$ 
with convolution operator $T^{\modulus}_{\lattice}$ on $\lattice$ defined by $\arithmeticFT{T^{\modulus}_{\lattice}f}(\freqpoint) = \mu^{\modulus}_{\torus}(\freqpoint) \arithmeticFT{f}(\freqpoint)$. 

\begin{MSW_transference} 
For $1<p<\infty$,
\begin{equation}\label{tranference_principle}
\norm{T^{\modulus}_{\lattice} f}_{\ell^p_{B_1} \to \ell^p_{B_2}} 
\lesssim \norm{T_{\euclideanspace} f}_{L^p_{B_1} \to L^p_{B_2}} .
\end{equation}
The implicit constant is independent of $B_1,B_2,p$ and $\modulus$.
\end{MSW_transference}

We will apply this lemma with $B_1 = B_2 = \ell^2(\lattice)$ in Lemma~\ref{lemma:maxHLop_arc_bound} of section~\ref{section:approximation_formula}.

\subsection{An $\ell^p$ inequality for $1/\modulus$-periodic multipliers on the torus}
\begin{lemma}[Magyar--Stein--Wainger]\label{lemma:MSW_periodic_inequality}
Suppose that $\mu(\toruspoint) = \sum_{\latticepoint \in \lattice} g({\latticepoint}) \bumpfxn(\toruspoint - \latticepoint/\modulus)$ is a multiplier on $\torus$ where $\bumpfxn$ is smooth and supported in $[-1/2\modulus,1/2\modulus]^\dimension$ with convolution operator $T$ on $\lattice$. Furthermore, assume that $g({\latticepoint})$ is $\modulus$-periodic ($g(\latticepoint_1) = g(\latticepoint_2)$ if $\latticepoint_1 \equiv \latticepoint_2 \mod{\modulus}$). For a $q$-periodic sequence, define the Fourier transform $\arithmeticFT{g}(\latticepoint) = \sum_{b \in \Zmod{\modulus}} g(b) e\left( \frac{\latticepoint \cdot b}{\modulus} \right)$. Then for $1\leq p \leq 2$,
\[
\norm{T}_{\ell^p(\lattice) \to \ell^p(\lattice)}
\lesssim \left( \sup_{m \in \Zmod{\modulus}} \lvert g({m}) \rvert \right)^{2-2/p} \left( \sup_{n \in \Zmod{\modulus}} \lvert \arithmeticFT{g}({n}) \rvert \right)^{2/p-1} 
\]
with implicit constants depending on $\bumpfxn$ and $p$, but not on $g$. 
\end{lemma}

\begin{proof}
Since $\bumpfxn(\xi-\latticepoint/\modulus)$ have disjoint supports for different $\latticepoint$, $\Lpnorm{\infty}{\torus}{\mu} \leq \sup_{\latticepoint} \absolutevalueof{a(\latticepoint)}$; this gives the $L^2$ bound by Plancherel's theorem. 
The $L^1$ bound follows from Minkowski's inequality because the kernel is $K(\latticepoint) = \arithmeticFT{\bumpfxn}(\latticepoint) \arithmeticFT{a}({\latticepoint})$ and $\lpnorm{1}{\torusFT{\bumpfxn}} \lesssim 1$ since $\bumpfxn$ is smooth. 
Interpolation finishes the lemma.
\end{proof}

\subsection{The main $\ell^2(\Z^\dimension)$ inequality}
An essential ingredient in the proof of Theorem \ref{thm:Wooley_max_fxn_bound} is the following inequality. Let $f$ be a function on $\lattice$, and for $\dyadicradius \leq \radius < 2\dyadicradius$, let $T_\radius$ be a convolution operator on $\lattice$ with multiplier $\mu_\radius(\freqpoint)= \int_{I} \alpha_\dyadicradius(t,\freqpoint) \eof{-\radius^\degree t}\; dt$ for a measurable subset $I$ of $\T$ or $\R$, and $\alpha_{\dyadicradius}$ is a measurable function on $\interval \times \torus$ depending on $\dyadicradius$, an additional parameter which is fixed for this discussion ($\dyadicradius$ will vary later). By Fourier inversion, we have
\[
\absolutevalueof{ T_\radius f} 
\leq \int_{I} \absolutevalueof{\arithmeticFT{\alpha}_\dyadicradius (t,\cdot) \convolvedwith f} \; dt 
.
\]
This last expression does not depend on $\radius$, but does depend on $\dyadicradius$. Therefore, 
\[
\disup{\dyadicradius}{\radius}{\absolutevalueof{ T_\radius f}} 
\leq \int_{I} \absolutevalueof{ \arithmeticFT{\alpha}_\dyadicradius (t,\cdot) * f } \, dt . 
\]
This allows us to estimate the $\ell^2(\Z^\dimension)$ norm of the dyadic maximal function.
\begin{lemma}[Main $\ell^2$ inequality]\label{lemma:main_ell2}
\[
\lpnorm{2}{\disup{\dyadicradius}{\radius}{T_\radius f}}
\leq \absolutevalueof{I} \cdot ||\alpha_\dyadicradius||_{L^{\infty}_{t, \freqpoint}(\interval \times \T^d)} \cdot \lpnorm{2}{f}
\]
where $|I|$ is the length of $I$. In what follows, the point below will be to bound $\alpha_\dyadicradius(t,\freqpoint)$ for $t \in \interval$, a major or minor arc, uniformly for $\freqpoint$ in $\torus$.
\end{lemma}

\section{Hypothesis $\hypothesis{\powerloss}$ and the general form of Theorem~\ref{thm:Wooley_max_fxn_bound}}\label{section:general_thm}
Our proof of Theorem~\ref{thm:Wooley_max_fxn_bound} relies heavily on bounds for exponential sums and oscillatory integrals. The necessary bounds for oscillatory integrals were reviewed in section~\ref{section:cont_max_fxns}. Note that the range of $L^p(\euclideanspace)$ spaces for the continuous $\degree$-spherical maximal functions in Proposition~\ref{prop:cont_max_fxn_bound_for_higher_order_spheres} uses sharp bounds for oscillatory integrals and is larger than possible for the arithmetic $\degree$-spherical maximal functions. Meanwhile the sharp bounds for our exponential sums are unknown. 
Motivated by Waring's problem and Vinogradov's mean value conjecture, we now describe our essential hypothesis on exponential sums which plays a similar role to the Bruna--Nagel--Wainger bounds \eqref{kspherical_decay} for the Fourier transform of the continuous $\degree$-spherical surface measure.  
\begin{Hypothesis}
Suppose that there exists integers $ 1 \leq \unit < \modulus \leq N$ relatively prime such that $\absolutevalueof{t-\unit/\modulus}~\leq~\modulus^{-2}$. Then
\[ 
\sum_{n=1}^N e(t n^\degree + \xi n) \lessapprox N (\modulus^{-1} + N^{-1} + \modulus N^{-\degree})^\powerloss 
\]
with implicit constants depending on $N$, but independent of $\xi$ and $\modulus$.
\end{Hypothesis}
%

\begin{rem}
The circle method in \cite{Davenport} shows that if Hypothesis $\hypothesis{\powerloss}$ is true, then $\nonzeroasymptoticshold \leq \max \inbraces{2 \degree, \degree/\powerloss}+1$. 
\end{rem}

Taking $\powerloss = 0$, we recover the trivial bound with a logarithmic-loss. The goal is to take $\powerloss$ as large as possible. Weyl gave the first non-trivial bound showing that we can take $\powerloss = 2^{1-\degree}$. So the hypothesis is not vacuous. There are many important works improving Weyl's bound, but the best asymptotic bound in $\degree$ is currently due to Wooley in \cite{Wooley_efficient_congruencing1}.
\begin{Wooley}
If $\degree \geq 3$, then $\hypothesis{\powerloss}$ is true with $\powerloss = \inbrackets{2 \degree \inparentheses{\degree-1}}^{-1}$, and if $\degree \geq 4$, then $\hypothesis{\powerloss}$ is true with $\powerloss = \inbrackets{2 \degree \inparentheses{\degree-2}}^{-1}$. 
\end{Wooley}
We now phrase The Approximation Formula and Theorem~\ref{thm:Wooley_max_fxn_bound} in terms of our hypothesis and record the best results currently available.
\begin{approximationlemma}
Let $\charfxnofarsphere(\latticepoint)$ be the characteristic function of $\arithmetichigherordersphere$ on $\lattice$. If hypothesis $\hypothesis{\powerloss}$ is true for some $0 < \powerloss < 1$, then for $\dimension>\max{\{ \degree(\degree+2), \degree/\powerloss\}}$, the $\lattice$ Fourier transform can be decomposed as 
\begin{equation}\label{approximation_formula}
\arithmeticFT{\charfxnofarsphere}(\xi) 
= \sum_{\modulus=1}^{\infty} \sum_{\unit \in \unitsmod{\modulus}} \eof{ \frac{\unit \radius^\degree}{\modulus} } \sum_{\latticepoint \in \lattice} \twistedGausssum{\latticepoint} \Psi(\modulus \xi - \latticepoint) \contFT{d\surfacemeasure_\radius}(\modulus \xi - \latticepoint) + \arithmeticFT{\error}(\xi)
\end{equation}
where the error term $\arithmeticFT{\error}$ is a multiplier term with convolution operator $\error$ satisfying 
\begin{equation}\label{error_bound}
\lpnorm{2}{\disup{\dyadicradius}{\radius}{\error}}
\lesssim \dyadicradius^{\dimension-\degree-\powersavings} 
\end{equation}
for some $\powersavings(\degree,\dimension, \powerloss)>0$. 
\end{approximationlemma}
%
%
\begin{thm}\label{thm:Vinogradov_max_thm}
Fix the degree $\degree \geq 3$. If $\hypothesis{\powerloss}$ is true for some $0 < \powerloss < 1$,  then $\maxop$ is bounded on $\ell^p(\lattice)$ for $\dimension > \max{\{ \degree(\degree+2), \degree/\powerloss\}}$ and $p> \max{ \{ \frac{\dimension}{\dimension -\degree/2\powerloss}, \frac{\dimension}{\dimension-\degree} \}}$.
\end{thm}

In particular, the Weyl bound allows us to take $\powerloss = 2^{1-\degree}$ for $\degree \geq 2$ while Wooley's sup bound allows us to take $\powerloss = \inbrackets{2\degree(\degree-2)}^{-1}$ for $\degree \geq 4$. 
Wooley's bounds improve on the classical Weyl bound for $\degree \geq 8$. 
It is unclear how small we can expect to take $\powerloss$. 
Therefore, the $\degree$-spherical maximal function is bounded on $\ell^p(\lattice)$ if 
\begin{itemize}
\item $p>\frac{\dimension}{\dimension - \degree 2^{\degree-2}}$ and $\dimension>\degree 2^{\degree-1}$ for $2\leq \degree \leq 7$
\item $p>\frac{\dimension}{\dimension - \degree^2(\degree-2)}$ and $\dimension > 2\degree^2 (\degree-2)$ for $\degree \geq 8$. 
\end{itemize}
At the bottom of page 196 in \cite{Montgomery}, Montgomery conjectures that one can take $\powerloss = \degree^{-1}$. 
Montgomery's conjecture implies the $\degree$-spherical maximal function is bounded on $\ell^p(\lattice)$ for $p>\frac{\dimension}{\dimension - \degree^2/2}$ and $\dimension>\degree(\degree+2)$. Note that this is still a order of $\degree$ away from the conjectured endpoint because a factor of $\degree$ is lost by using sup bounds. In Waring's problem the loss of a factor of $\degree$ by sup bounds is overcome in by using mean values of exponential sums. 
Regrettably, our method does not exploit this technique. 
%
%
We state similar improvements to the pointwise ergodic theorem in section~\ref{section:ergodic}. 

\section{The dyadic maximal operator}\label{section:dyadic}
In this section we prove that $\dyadicmaxop$ is uniformly bounded in $\dyadicradius$ on $\ell^p(\lattice)$ for a range of $\ell^p(\lattice)$-spaces depending on Hypothesis $\hypothesis{\powerloss}$. 
Our analysis follows closely the analysis in \cite{Magyar_dyadic} and \cite{AS}; we include the proof for completeness. 
Theorem~\ref{thm:dyadic} will be used in the proof of Theorem~\ref{thm:Vinogradov_max_thm} in section~\ref{section:maximal_theorem_proof}. 
\begin{thm}\label{thm:dyadic}
If Hypothesis $\hypothesis{\powerloss}$ is true for some $0<\powerloss<1$, then for all $\dyadicradius > 0$, $\dyadicmaxop$ is uniformly bounded in $\dyadicradius$ on $\ell^p(\lattice)$ for $\dimension > \max{\inbraces{2\degree, \degree/\powerloss }}$ and  $\max{\inbraces{\frac{\dimension}{\dimension-\degree}, \frac{\dimension}{\dimension -\degree/2\powerloss}} } < p \leq 2$.
\end{thm}

In section~\ref{section:circle_method_decomposition} we follow the circle method paradigm and decompose our operators $\avgop$ into pieces corresponding to major and minor arcs. In section~\ref{section:dyadic_major_arcs} we bound each major arc piece in Lemma~\ref{lemma:dyadic_major_arc_bound}, and in section~\ref{section:dyadic_minor_arcs} we bound each minor arc piece in Lemma~\ref{lemma:dyadic_minor_arc_bound}. Theorem~\ref{thm:dyadic} follows immediately from Corollaries~\ref{cor:dyadic_major_arcs_bound} and \ref{cor:dyadic_minor_arcs_bound}.
%

\subsection{The circle method decomposition}\label{section:circle_method_decomposition}
%
If the dimension is sufficiently large, say $\dimension \geq \nonzeroasymptoticshold$, then $\numberoflatticepoints(\radius) \uptoconstants \radius^{\dimension-\degree}$. This allows us to redefine the averages $\avgop$ to be
\[
\averagefxn(\latticepoint) 
= \radius^{\degree-\dimension} \sum_{n \in \arithmetichigherordersphere} \fxn(\latticepoint-n)
\] 
at the expense of a constant. $\avgop$ is a convolution operator with multiplier $\latticeFT{\avgop}(\freqpoint) = \radius^{\degree-\dimension} a_{\radius}(\freqpoint)$ where 
\[
a_{\radius}(\freqpoint) 
:=  \sum_{\latticepoint \in \arithmetichigherordersphere} \eof{\latticepoint \cdot \freqpoint} 
\] 
is the Fourier transform of the characteristic function of the set of lattice points on the $\degree$-sphere $\arithmeticsphere{\radius}$. Using the orthogonality relation 
\begin{equation}\label{orthogonality_relation}
\delta(n) 
= \int_0^1 \eof{m t} \, dt
\end{equation} 
for $n \in \Z$, we rewrite $a_{\radius}$ as 
\[
a_{\radius}(\freqpoint) 
= \int_{0}^{1} \sum_{|\latticepoint| \leq 2\dyadicradius} \eof{(\absolutevalueof{\latticepoint}^\degree-\radius^\degree)t + \latticepoint \cdot \freqpoint} \, dt 
\] 
for any fixed $\dyadicradius > \radius$. 
The sum above is over lattice points in a ball of radius $2 \radius$ since the $\degree$-sphere is contained in the ball of radius $2 \radius$ for any degree $\degree$. 
Note that 
\[
\avgop \fxn(\latticepoint) 
= \int_0^1 \eof{(|\cdot|^\degree - \radius^\degree) t} \indicator{B_{2\dyadicradius}} \convolvedwith \fxn (\latticepoint) \, dt 
\] 
where $\indicator{B_{\dyadicradius}}$ is the characteristic function of the ball of radius $\dyadicradius$. 

%
Given $X \in \R^{+}$, define the Farey sequence $\Fareysequence{X}$ of level $X$ as the set $\inbraces{\Fareyfraction : 1 \leq \unit \leq \modulus \leq X, (\unit,\modulus)=1}$ written in increasing order. The Farey sequence allows us to partition the unit interval $[0,1)$ in the following way. For each $\Fareyfraction \in \Fareysequence{X}$, suppose that $a_1/q_1 < a/q < a_2/q_2$ are neighbors, and let $\overline{\arc}:= \left[ \frac{a+a_1}{q+q_1}, \frac{a+a_2}{q+q_2} \right)$. This does not make sense at the endpoints, 0 and 1, and we include these points by letting $\overline{I(0)}=\overline{I(1)}:= \left[0,1/\floor{X} \right) \cup \left[ \frac{\floor{X}-1}{\floor{X}},1 \right)$. Then $[0,1)$ is the disjoint union of arcs $\overline{\arc}$ for ${\Fareyfraction \in \Fareysequence{X}}$. These are called arcs, a term from the original version of the circle method. We will need to make each arc (almost) symmetric about 0, so we shift $\overline{\arc}$ by $\Fareyfraction$ to get the arc $\arc := \left[ \frac{a+a_1}{q+q_1}, \frac{a+a_2}{q+q_2} \right) - \Fareyfraction = \left[ \frac{-1}{\modulus(\modulus+\modulus_1)}, \frac{1}{\modulus(\modulus+\modulus_2)} \right)$. For $\unit/\modulus \in \Fareysequence{X}$, $X \leq \modulus+\modulus_1, \modulus+\modulus_2 \leq 2X$.
Threfore, each arc has length $1/qX \leq |\arc| \leq 2/qX$. 
See \cite{HardyWright} for more details on the Farey sequence and Diophantine approximation.
%
We make a Farey dissection of level $\WaringFareylevel$ on $\unitinterval$. This decomposes $[0,1]$ into the disjoint union of \emph{arcs} $\overline{\arc}$ for $1 \leq \unit < \modulus < \dyadicradius^\degree$ and $\unit \in \unitsmod{\modulus}$. The Farey dissection induces the following decomposition on $a_{\radius}$: 
\[
a_{\radius}(\freqpoint) 
= \sum_{\modulus < \dyadicradius^\degree} \sum_{\unit \in \unitsmod{\modulus}} \int_{\overline{\arc}} \eof{-t \radius^\degree} \sum_{|\latticepoint| \leq 2\dyadicradius} \eof{\absolutevalueof{\latticepoint}^\degree t + \latticepoint \cdot \freqpoint} \, dt 
. 
\]
We isolate each piece to define for an arc $\arc$, 
\[
a_{\radius}^{\Fareyfraction}(\freqpoint) 
:= \int_{\overline{\arc}} \eof{-t \radius^\degree} \sum_{|\latticepoint| \leq 2\dyadicradius} \eof{\absolutevalueof{\latticepoint}^\degree t + \latticepoint \cdot \freqpoint} \, dt 
. 
\]
%
By translating $\overline{\arc}$ to $\arc := \overline{\arc}-\unit/\modulus$, we find that 
\begin{equation}\label{avgop_arc_definition}
a^{\unit/\modulus}_\radius(\freqpoint) 
= \eof{\radius^\degree \unit/\modulus} \int_{\arc} \eof{-t \radius^\degree} \sum_{|\latticepoint| \leq 2\dyadicradius} \eof{\absolutevalueof{\latticepoint}^\degree t + \latticepoint \cdot \freqpoint} \, dt 
. 
\end{equation}

We now define the major and minor arcs; let $\majorarcs := \{ \arc : 1\leq \unit < \modulus < \dyadicradius , \unit \in \unitsmod{\modulus} \}$ be the major arcs and $\minorarcs := [0,1] \setminus \majorarcs$ be the minor arcs. This splits the multiplier into major and minor arc pieces. Define the major arc multiplier 
\[
a^{Major}_{\radius} 
:= \sum_{\modulus < \WaringMajorarclevel} \sum_{\unit \in \unitsmod{\modulus}} a^{\unit/\modulus}_\radius 
\]
and minor arc multiplier 
\[
a^{minor}_{\radius} 
:= \sum_{\WaringMajorarclevel \leq \modulus < \WaringFareylevel} \sum_{\unit \in \unitsmod{\modulus}} a^{\unit/\modulus}_\radius 
\]
so that 
\[
a_{\radius} 
= a^{Major}_{\radius} + a^{minor}_{\radius} .
\]
$\avgop^{Major}$ and $\avgop^{minor}$ are their respective convolution operators normalized so that $\latticeFT{\avgop^{Major}}(\freqpoint) = \radius^{\degree-\dimension} a^{Major}_{\radius}$ and similarly for $\avgop^{minor}$. 
%
%
%
Fix $\dyadicradius >0$. 
By the triangle inequality, we reduce to bounding the dyadic maximal major arc operator and the dyadic maximal minor arc operator: 
\[
\lpnorm{p}{\disup{\dyadicradius}{\radius}{\absolutevalueof{\avgop \fxn}}} 
\leq \lpnorm{p}{\disup{\dyadicradius}{\radius}{\absolutevalueof{ \avgop^{Major} \fxn} } } + \lpnorm{p}{\disup{\dyadicradius}{\radius}{\absolutevalueof{ \avgop^{minor} \fxn} } } 
. 
\]
%

\subsection{Major arcs bounds for the dyadic maximal operator}\label{section:dyadic_major_arcs}
We begin our analysis of the dyadic maximal function by studying the dyadic maximal function of a major arc piece. It will be convenient to replace the exponential sum $\sum_{\absolutevalueof{\latticepoint} \leq 2\dyadicradius} \eof{\absolutevalueof{\latticepoint}^\degree t + \latticepoint \cdot \freqpoint}$ with a sharp cut-off by the smoothed exponential sum $\sum_{\latticepoint\in \lattice} \eof{\absolutevalueof{\latticepoint}^\degree z + \latticepoint \cdot \freqpoint}$, for complex $z = t+i \epsilon$ with $t$ in a major arc $\arc$ and $\epsilon > 0$. We recognize this as the Fourier transform on $\lattice$ of an analytic function and use Poisson summation to estimate the smoothed exponential sum. This will lead us to the following lemma which is the main result of this section. 
%
\begin{lemma}\label{lemma:dyadic_major_arc_bound} 
If $t$ is on a major arc $\arc \in \majorarcs$ and $1 \leq p \leq 2$, then 
\[
\lpnorm{p}{\disup{\dyadicradius}{\radius}{\absolutevalueof{ \avgop^{\Fareyfraction} \fxn} } } 
\lessapprox \modulus^{-\frac{\dimension}{\degree}(2-2/p)} \lpnorm{p}{\fxn} 
.
\]
\end{lemma}
Summing over the major arcs, we conclude 
\begin{cor}\label{cor:dyadic_major_arcs_bound} 
If $\dimension > 2\degree$ and $p > \frac{\dimension}{\dimension-\degree}$, then 
\[
\lpnorm{p}{\disup{\dyadicradius}{\radius}{\absolutevalueof{\avgop^{Major} \fxn} } } 
\lesssim \lpnorm{p}{\fxn} 
.
\]
\end{cor}
%
Before we can prove Lemma~\ref{lemma:dyadic_major_arc_bound}, we need to state two results. The first result is the following estimate of Hardy -- see equations (16) and (17) on pg. 13 of \cite{Magyar_dyadic}.
\begin{prop}\label{prop:freq_localization}
Let $h_z(x) = \eof{\absolutevalueof{x}^\degree z}$ and $\contFT{h_z}(\freqpoint)$ be its Fourier transform on $\euclideanspace$. 
For all $\freqpoint \in \R^\dimension$,
\begin{enumerate}
\item $ \contFT{h_z}(\freqpoint) \lesssim |z|^{-\frac{\dimension}{\degree}} $
\item $ \contFT{h_z}(\freqpoint-\latticepoint/\modulus ) \lesssim |z|^{-\frac{\dimension}{2\degree-2}}|\freqpoint-\latticepoint/\modulus |^{-\dimension \frac{\degree-2}{2\degree-2}} e^{-K \cdot |\modulus \freqpoint-\latticepoint|^\frac{\degree}{\degree-1}} $ for some constant $K$ depending only on dimension $\dimension$ and degree $\degree$. 
\end{enumerate}
\end{prop}
The proof of Proposition~\ref{prop:freq_localization} is technical, using the method of steepest descent; we refer the reader to \cite{HL} for a proof. The second result is a $\Zmod{\modulus}$-analogue of Hardy's estimate for Gauss sums and is due to Hua. 
\begin{Huabound} 
For all $\unit \in \unitsmod{\modulus}$ and $\latticepoint \in \lattice$, we have 
\begin{equation}
\twistedGausssum{\latticepoint} 
\lessapprox \modulus^{-\frac{\dimension}{\degree}}  
\end{equation}
where the implicit bounds depend on the degree $\degree$ and dimension $\dimension$, but are independent of $\unit$, $\modulus$ and $\latticepoint$. 
\end{Huabound}
For a proof of this result in one-dimension, we refer the reader to Theorem 7.1 on page 112 of \cite{Vaughan}. The $\dimension$-dimensional version is obtained by taking the $\dimension$-fold product of the one-dimensional version.

\begin{proof}[Proof of Lemma~\ref{lemma:dyadic_major_arc_bound}]
We start by smoothing the sharp cutoff $\indicator{B_{2\dyadicradius}}$ by factors of $e^{-2\pi \epsilon |\latticepoint|^\degree} = \eof{i \epsilon |\latticepoint|^\degree}$ for $\epsilon>0$ to get 
\[
a_{\radius}^{\Fareyfraction}(\freqpoint) 
= e^{2\pi \epsilon \radius^\degree} \eof{\unit \radius^\degree/\modulus} \int_{\arc} \sum_{\latticepoint\in \lattice} \eof{|\latticepoint|^\degree (t+i\epsilon) -t \radius^\degree + \latticepoint \cdot \freqpoint} \, dt 
.
\]
Note that the sum is now over $\lattice$ rather than $B_{2\dyadicradius}$; this is possible due to the orthogonality relation \eqref{orthogonality_relation} and exponential decay of $\eof{i\epsilon|\latticepoint|^\degree}$ as $|\latticepoint| \to \infty$. 
Let $z = t+i\epsilon$ and $h_z(\latticepoint) := \eof{|\latticepoint|^\degree z}$ so that $\latticeFT{h_z}(\freqpoint) = \sum_{\latticepoint\in \lattice} \eof{|\latticepoint|^\degree z + \latticepoint \cdot \freqpoint}$ and 
\[
a_{\radius}^{\Fareyfraction}(\freqpoint) 
= e^{2\pi \epsilon \radius^\degree} \eof{\unit \radius^\degree/\modulus} \int_{\arc} \eof{-t \radius^\degree} \latticeFT{h_z}(\freqpoint) \, dt .
\]
We choose $\epsilon = \dyadicradius^{-\degree}$ so that $e^{2\pi \epsilon \radius^\degree} \uptoconstants 1$ for $\dyadicradius  \leq \radius < 2\dyadicradius$. 
Minkowski's inequality implies
\[
\lpnorm{p}{\disup{\dyadicradius}{\radius}{ \absolutevalueof{\avgop^{\Fareyfraction} \fxn} }} 
\lesssim \dyadicradius^{\degree-\dimension} \int_{\arc} \lpnorm{p}{h_z \convolvedwith \fxn} \, dt 
\]
for each arc $\arc$. Therefore we reduce to bounding $\lpnorm{p}{h_z \convolvedwith \fxn}$ for $t \in \arc \in \majorarcs$. 
We will bound $\lpnorm{p}{h_z \convolvedwith \fxn}$ by  interpolation of bounds for $\lpnorm{1}{h_z \convolvedwith \fxn}$ and $\lpnorm{2}{h_z \convolvedwith \fxn}$. 

Since $\lpnorm{1}{h_z} \lesssim \epsilon^{-\dimension/\degree}$, we apply Fubini's theorem to easily deduce the $\ell^1(\lattice)$ bound: 
\begin{equation}\label{smooth_ell1_bound}
\lpnorm{1}{h_z \convolvedwith \fxn} 
\lesssim \epsilon^{-\dimension/\degree} \lpnorm{1}{\fxn} 
. 
\end{equation}
%
For the bound on $\ell^2(\lattice)$, Plancherel's theorem implies 
\[
\lpnorm{2}{h_z \convolvedwith \fxn} 
\leq \Lpnorm{\infty}{\torus}{\latticeFT{h_z}} \lpnorm{2}{\fxn} 
. 
\]
Therefore we wish to bound $\Lpnorm{\infty}{\torus}{\latticeFT{h_z}}$ for $t \in \arc$. 
We apply the Poisson summation formula, Proposition~\ref{prop:freq_localization} and Hua's inequality for Gauss sums to deduce 
\begin{align*}
\latticeFT{h_z}(\freqpoint) 
&= \sum_{\latticepoint} \twistedGausssum{\latticepoint} \widetilde{h_z}(\latticepoint/\modulus -\freqpoint) \\
& \lessapprox \modulus^{-\frac{\dimension}{\degree}} \inparentheses{ \absolutevalueof{z}^{-\frac{\dimension}{\degree}} + \absolutevalueof{z}^{-\frac{\dimension}{2\degree-2}} \modulus^{\frac{\dimension(\degree-2)}{2\degree-2}} } \\
& = (\epsilon \modulus)^{-\frac{\dimension}{\degree}} \inparentheses{ \absolutevalueof{i+t/\epsilon}^{-\frac{\dimension}{\degree}} + \absolutevalueof{i+t/\epsilon}^{-\frac{\dimension}{2\degree-2}} (\epsilon \modulus)^{\frac{\dimension(\degree-2)}{2\degree-2}} } 
. 
\end{align*}
Therefore, 
\begin{equation}\label{smooth_ell2_bound}
\lpnorm{2}{h_z \convolvedwith \fxn} 
\leq (\epsilon \modulus)^{-\frac{\dimension}{\degree}} \inparentheses{ \absolutevalueof{i+t/\epsilon}^{-\frac{\dimension}{\degree}} + \absolutevalueof{i+t/\epsilon}^{-\frac{\dimension}{2\degree-2}} (\epsilon \modulus)^{\frac{\dimension(\degree-2)}{2\degree-2}} } \lpnorm{2}{\fxn} 
\end{equation}
The implicit constants depend only on the dimension $\dimension$ and degree $\degree$. 
Interpolating \eqref{smooth_ell1_bound} with \eqref{smooth_ell2_bound}, we find 
\begin{equation}\label{smooth_ellp_bound}
\lpnorm{p}{h_z \convolvedwith \fxn} 
\lessapprox \epsilon^{-\dimension/\degree} 
  \inbrackets{ \modulus^{-\frac{\dimension}{\degree}} \inparentheses{ \absolutevalueof{i+t/\epsilon}^{-\frac{\dimension}{\degree}} + \absolutevalueof{i+t/\epsilon}^{-\frac{\dimension}{2\degree-2}} (\epsilon \modulus)^{\frac{\dimension(\degree-2)}{2\degree-2}} } 
  }^{2-2/p}
\end{equation}
for $1 \leq p \leq 2$. 
By integrating \eqref{smooth_ellp_bound} over $\arc$ using that $\epsilon=\dyadicradius^{-\degree}$ and $\modulus \leq \WaringMajorarclevel$, we conclude the proof. 
\end{proof}

\subsection{Minor arcs bounds for the dyadic maximal operator}\label{section:dyadic_minor_arcs}
In this section, we now consider the minor arcs. In Lemma~\ref{lemma:dyadic_minor_arc_bound} we prove a bound analogous to Lemma~\ref{lemma:dyadic_major_arc_bound}. 
Since $t$ is on a minor arc, there exists $\dyadicradius \leq \modulus < \dyadicradius^{\degree-1}$ and $\unit \in \Zmod{\modulus}$ such that $\absolutevalueof{t-\unit/\modulus} \leq 1/\modulus \dyadicradius^{\degree-1}$. Hypothesis $\hypothesis{\powerloss}$ then implies that 
\begin{equation}
\sum_{|\latticepoint| \leq 2\dyadicradius} \eof{\absolutevalueof{\latticepoint}^\degree t + \latticepoint \cdot \freqpoint} 
\lessapprox \dyadicradius^{\dimension (1-\powerloss)} 
.
\end{equation}
We use this bound to bound the dyadic maximal function for a minor arc piece in the following lemma. 
%
\begin{lemma}\label{lemma:dyadic_minor_arc_bound} 
If Hypothesis $\hypothesis{\powerloss}$ is true for some $0<\powerloss<1$, then on a minor arc $\arc \in \minorarcs$, we have 
\[
\lpnorm{p}{\disup{\dyadicradius}{\radius}{\absolutevalueof{\avgop^{\Fareyfraction} \fxn}}} 
\lessapprox \absolutevalueof{\arc} \dyadicradius^{\degree-\dimension \powerloss (2-2/p)} \lpnorm{p}{\fxn} 
\]
for any $1 \leq p \leq 2$ and 
where the implicit constants are independent of $\dyadicradius > 0$.
\end{lemma}
Summing over the minor arcs, we conclude 
\begin{cor}\label{cor:dyadic_minor_arcs_bound} 
If Hypothesis $\hypothesis{\powerloss}$ is true for some $0<\powerloss<1$ and $1 \leq p \leq 2$, 
then 
\begin{equation}\label{minor_arcs_max_fxn_bound} 
\lpnorm{p}{\disup{\dyadicradius}{\radius}{\absolutevalueof{\avgop^{minor} \fxn}}} 
\lessapprox \dyadicradius^{\degree - \dimension \powerloss (2-2/p)} \lpnorm{p}{\fxn} 
\end{equation}
where the implicit constants are independent of $\dyadicradius > 0$.
\end{cor}
\begin{rem}
Note that by Corollary~\ref{cor:dyadic_minor_arcs_bound}, if $\degree < \dimension \powerloss (2-2/p)$ and $1 \leq p \leq 2$, then 
\[
\lpnorm{p}{\disup{\dyadicradius}{\radius}{\absolutevalueof{\avgop^{minor} \fxn}}} 
\lesssim \lpnorm{p}{\fxn} 
. 
\] 
This occurs when $\dimension > \degree/\powerloss$ and $\frac{\dimension}{\dimension-\degree/2\powerloss} < p \leq 2$. 
\end{rem}
%
\begin{proof}[Proof of Lemma~\ref{lemma:dyadic_minor_arc_bound}]
Similar to \eqref{smooth_ell1_bound}, Fubini's theorem implies the $\ell^1(\lattice)$ bound: 
\begin{equation}\label{rough_ell1_bound}
\lpnorm{1}{\indicator{B_{2\dyadicradius}} \convolvedwith \fxn} 
\lesssim \dyadicradius^{\dimension} \lpnorm{1}{\fxn} 
. 
\end{equation}
By Minkowski's inequality, the  $\ell^1(\lattice)$ bound \eqref{rough_ell1_bound} implies 
\begin{equation}
\lpnorm{1}{\disup{\dyadicradius}{\radius}{\absolutevalueof{ \avgop^{\Fareyfraction} \fxn} } } 
\lesssim \absolutevalueof{\arc} \dyadicradius^{\degree} \lpnorm{1}{\fxn} 
. 
\end{equation}
By interpolation, we are reduced to proving the $\ell^2(\lattice)$ bound of Lemma~\ref{lemma:dyadic_minor_arc_bound}. 
Recall that 
\[
a^{\Fareyfraction}_{\radius}(\freqpoint) 
= \eof{\radius^\degree \Fareyfraction} \int_{\arc} \sum_{|\latticepoint| \leq 2\dyadicradius} \eof{(|\latticepoint|^\degree-\radius^\degree)t + \latticepoint \cdot \freqpoint} \; dt 
. 
\]
Hypothesis $\hypothesis{\powerloss}$ implies 
\[
\absolutevalueof{a^{\Fareyfraction}_\radius(\freqpoint)} 
\lessapprox \absolutevalueof{\arc} \dyadicradius^{\dimension \inparentheses{ 1-\powerloss }}
\]
uniformly for $t \in \arc$ and $\freqpoint \in \torus$. The \emph{main $\ell^2$ inequality} (Lemma~\ref{lemma:main_ell2}) implies 
\[ 
\lpnorm{2}{\disup{\dyadicradius}{\radius}{|\avgop^{\Fareyfraction} \fxn|}} 
\lessapprox \absolutevalueof{\arc} \dyadicradius^{\degree-\dimension \powerloss} \lpnorm{2}{\fxn} 
. 
\]
\end{proof}
%
%
\begin{rem}
The proof of Lemma~\ref{lemma:dyadic_minor_arc_bound} is quite simple when using Hypothesis $\hypothesis{\powerloss}$. However, the difficulty is in proving Hypothesis $\hypothesis{\powerloss}$ is true for small $\powerloss$. Finding the smallest possible $\powerloss$ for which Hypothesis $\hypothesis{\powerloss}$ is true has a rich history in number theory and is in the province of Waring's problem. 
\end{rem}

\section{The Approximation Formula}\label{section:approximation_formula}
The goal of this section is to prove The Approximation Formula of section~\ref{section:general_thm}. Continuing with the circle method paradigm, we further analyze the multipliers $a_{\radius}^{\Fareyfraction}$ on the major arcs to pull out a main term which we identify as a weighted piece of the Gelfand--Leray measure on the continuous $\degree$-sphere. We show that the remainder is a well-bounded error term. 

Fix $\arc$ to be a major arc. If we write
\[
J_{\radius}(\freqpoint- \latticepoint/\modulus)
:= \int_{\arc}  \contFT{h_z}(\latticepoint/\modulus -\freqpoint) \eof{-\radius^\degree t} \; dt 
, 
\]
then by \eqref{avgop_arc_definition}, we have  
\[
a^{\Fareyfraction}_{\radius}(\freqpoint) 
= e^{2\pi \epsilon \radius^\degree} \eof{\radius^\degree \Fareyfraction} \sum_{\latticepoint} \twistedGausssum{\latticepoint} J_{\radius}(\freqpoint- \latticepoint/\modulus ) 
. 
\]
Recall that $\Psi$ is a smooth function supported in $\unitcube$ such that $\Psi(x) = 1$ for $x \in \inbrackets{-1/4, 1/4}^{\dimension}$. 
Introduce the approximating multipliers $b^{\Fareyfraction}_\radius$ and $\HLmultiplier^{\Fareyfraction}_{\radius}$ defined as follows: 
\begin{equation}\label{first_approximation_definition}
b^{\Fareyfraction}_\radius(\freqpoint) 
:= e^{2\pi \epsilon \radius^\degree} \eof{\radius^\degree \Fareyfraction} \sum_{\latticepoint} \twistedGausssum{\latticepoint} \Psi(\modulus \freqpoint-\latticepoint) J_{\radius}(\freqpoint- \latticepoint/\modulus )
\end{equation}
and 
\begin{equation}\label{HLop_definition}
\HLmultiplier^{\Fareyfraction}_{\radius}(\freqpoint) 
:= e^{2\pi \epsilon \radius^\degree} \eof{\radius^\degree \Fareyfraction} \sum_{\latticepoint} \twistedGausssum{\latticepoint} \Psi(\modulus \freqpoint-\latticepoint) I_{\radius}(\freqpoint- \latticepoint/\modulus ) 
\end{equation}
where 
\[
I_{\radius}(\freqpoint) := \int_{\R} \contFT{h_z}(\freqpoint) \eof{-\radius^\degree t} \, dt 
\]
is obtained from $J_{\radius}$ by extending the range of integration from $\arc$ to $\R$. Since $\Psi$ is supported in $\unitcube$, there is a unique non-zero term in each sum of \eqref{first_approximation_definition} and \eqref{HLop_definition}. 
As proved in \cite{MSW} for the sphere and \cite{Magyar_ergodic} for positive homogeneous forms, 
\[
I_{\radius}(\freqpoint) 
= \radius^{\dimension-\degree} e^{-\pi \epsilon \radius^\degree} \contFT{d\higherorderspheremeasure_\radius}(\freqpoint)
\]
where $d\higherorderspheremeasure_\radius$ is the Gelfand--Leray measure on the hypersurface $\higherordersphere$ with the normalization
\[
\volumeof(\unithigherordersphere) 
= \frac{\Gamma(\frac{\degree+1}{\degree})^\dimension}{\Gamma(\frac{\dimension}{\degree})}
.
\]
This normalization is chosen to match the normalization of the averaging operators $\avgop$ and the asymptotic formula in Waring's problem. 
We rewrite $\HLmultiplier^{\Fareyfraction}_{\radius}$ as 
\begin{equation}\label{HL_multiplier_defn}
\HLmultiplier^{\Fareyfraction}_{\radius}(\freqpoint) 
= \radius^{\dimension-\degree} \eof{\radius^\degree \Fareyfraction} \sum_{\latticepoint} \twistedGausssum{\latticepoint} \Psi(\modulus \freqpoint-\latticepoint) \, \contFT{d\higherorderspheremeasure_{\radius}}(\latticepoint/\modulus -\freqpoint) 
. 
\end{equation}
The multipliers $b^{\Fareyfraction}_\radius$ and $\HLmultiplier^{\Fareyfraction}_{\radius}$ have convolution operators $\radius^{\dimension-\degree} B^{\Fareyfraction}_\radius$ and $\radius^{\dimension-\degree} \HLop^{\Fareyfraction}$, respectively; the factor of $\radius^{\dimension-\degree}$ is so that $B^{\Fareyfraction}_\radius$ and $\HLop^{\Fareyfraction}$ have the same normalization as $\avgop^{\Fareyfraction}$. 
Similar to $\avgop^{Major}$, define multipliers $b_{\radius}^{Major}$ and $c_{\radius}^{Major}$ with convolution operators $\radius^{\dimension-\degree} B^{Major}_{\radius}$ and $\radius^{\dimension-\degree} \HLop^{Major}$, respectively, by
\begin{align*} 
b^{Major}_{\radius} 
& := \sum_{\modulus <\WaringMajorarclevel} \sum_{\unit \in \unitsmod{\modulus}} b^{\Fareyfraction}_{\radius} \\
\HLmultiplier^{Major}_{\radius} 
& := \sum_{\modulus <\WaringMajorarclevel} \sum_{\unit \in \unitsmod{\modulus}} \HLmultiplier^{\Fareyfraction}_{\radius} 
. 
\end{align*}


We are now ready to state the main result of this section which says that $\avgop^{Major}$ aprroximates $\HLop^{Major}$ well in the $\ell^2(\lattice)$-norm. 
\begin{dyadic_major_arc_approximation}
If $\dyadicradius > 0$, then for all $\fxn \in \ell^2(\lattice)$, we have 
\begin{equation}\label{dyadic_major_arc_bound}
\lpnorm{2}{\disup{\dyadicradius}{\radius}{\absolutevalueof{A^{Major}_{\radius}\fxn-\HLop^{Major} \fxn} } } 
\lessapprox \dyadicradius^{\degree+2-\frac{\dimension}{\degree}} \lpnorm{2}{\fxn} 
\end{equation}
where the implicit constants only depend on dimension and degree. 
\end{dyadic_major_arc_approximation}
By the triangle inequality we reduce to studying:  
\[
\lpnorm{2}{\disup{\dyadicradius}{\radius}{|A^{\Fareyfraction}_{\radius}\fxn-B^{\Fareyfraction}_{\radius}\fxn |}}
\]
and
\[
\lpnorm{2}{\disup{\dyadicradius}{\radius}{|B^{\Fareyfraction}_{\radius}\fxn-{\HLop}^{\Fareyfraction}\fxn |}}
\]
for each major arc. We now prove a power-saving bound for each approximation and sum over the major arcs to prove the dyadic major arc approximation lemma. For the first approximation we prove the following bound for each major arc. 
\begin{lemma}\label{lemma:first_dyadic_aprroximation_bound}
For all $\dyadicradius>0$, we have 
\[
\lpnorm{2}{\disup{\dyadicradius}{\radius}{|A^{\Fareyfraction}_{\radius}\fxn-B^{\Fareyfraction}_{\radius}\fxn|} } 
\lessapprox \dyadicradius^{\degree - \frac{\dimension}{2}} \modulus^{\frac{\dimension}{2}-\frac{\dimension}{\degree}} \lpnorm{2}{\fxn} 
\]
where the implicit constants are independent of $\dyadicradius$. 
\end{lemma}
Summing over the major arcs, we conclude 
\begin{cor}\label{cor:first_step_major_arcs_approximation_bound}
For all $\dyadicradius > 0$,
\[
\lpnorm{2}{\disup{\dyadicradius}{\radius}{|A^{Major}_{\radius}f-B^{Major}_{\radius} \fxn|} } 
\lessapprox \dyadicradius^{\degree+2-\frac{\dimension}{\degree}} \lpnorm{2}{\fxn} 
\]
where the implicit constants are independent of $\dyadicradius$. 
\end{cor}
%
\begin{proof}[Proof of Lemma~\ref{lemma:first_dyadic_aprroximation_bound}]
By the main $\ell^2$ inequality, we reduce to proving 
\[
\sum_{\latticepoint \in \lattice}\twistedGausssum{\latticepoint} (1-\Psi(\modulus\freqpoint-\latticepoint)) \contFT{h_z}(\freqpoint-\latticepoint/\modulus) 
\lessapprox \dyadicradius^{\dimension/2} \modulus^{\frac{\dimension}{2}-\frac{\dimension}{\degree}} 
. 
\] 
The support $\Psi$ implies that we can restrict the summation to $\absolutevalueof{\modulus \freqpoint-\latticepoint} > 1/2$. Hua's bound for Gauss sums and Proposition~\ref{prop:freq_localization} imply 
\begin{align*}
\sum_{\latticepoint}\twistedGausssum{\latticepoint} (1-\Psi(\modulus\freqpoint-\latticepoint)) \contFT{h_z}(\freqpoint-\latticepoint/\modulus) 
& \lesssim \sup_{\latticepoint} |\twistedGausssum{\latticepoint}| \sum_{|\modulus\freqpoint-\latticepoint|>1/2} | \contFT{h_z}(\freqpoint-\latticepoint/\modulus) | \\
& \lessapprox \modulus^{-\dimension/\degree} \sum_{|\modulus\freqpoint-\latticepoint|>1/2} |z|^{-\frac{\dimension}{2\degree-2}}|\freqpoint-\latticepoint/\modulus|^{-\dimension \frac{\degree-2}{2\degree-2}} e^{-K \cdot |\modulus\freqpoint-\latticepoint|^\frac{\degree}{\degree-1}} \\
& \lesssim \dyadicradius^{\dimension/2} \modulus^{\frac{\dimension}{2}-\frac{\dimension}{\degree}} 
. 
\end{align*}
\end{proof}
We now handle the second approximation. 
\begin{lemma}\label{lemma:HLop_dyadic_aprroximation_bound}
For all $\dyadicradius > 0$,
\[
\lpnorm{2}{\disup{\dyadicradius}{\radius}{|B^{\Fareyfraction}_{\radius} \fxn - \HLop^{\Fareyfraction} \fxn | } } 
\lessapprox \modulus \WaringMajorarclevel^{-\frac{\dimension}{\degree}+1} \lpnorm{2}{\fxn}
.
\]
\end{lemma}
Summing over the major arcs, we conclude 
\begin{cor}\label{cor:HLop_dyadic_aprroximation_bound}
\[
\sum_{\modulus =1}^\WaringMajorarclevel \sum_{\unit \in \unitsmod{\modulus}} \lpnorm{2}{\disup{\dyadicradius}{\radius}{| B^{\Fareyfraction}_{\radius} \fxn - \HLop^{\Fareyfraction} \fxn |}} 
\lessapprox \dyadicradius^{3-\frac{\dimension}{\degree}} \lpnorm{2}{\fxn} 
. 
\]
\end{cor}
\begin{proof}[Proof of Lemma~\ref{lemma:HLop_dyadic_aprroximation_bound}]
By the main $\ell^2$ inequality, we reduce to proving 
\[
\sum_{\latticepoint\in \lattice} \twistedGausssum{\latticepoint} \Psi(\modulus \freqpoint-\latticepoint)\left( I_{\radius} - J_{\radius} \right)(\freqpoint) 
\lessapprox \WaringMajorarclevel^{\dimension-\degree-\frac{\dimension}{\degree}+1} \modulus^{-1} 
. 
\]
The size of $\arc$ is approximately $1/\modulus \WaringFareylevel$; this implies 
\begin{align*} 
\inbrackets{I_{\radius} - J_{\radius}} (\freqpoint-\latticepoint/\modulus) 
& = \int_{\R\setminus \arc} \eof{-\radius^\degree t} \contFT{h_z}(\freqpoint) \; dt \\
& \lesssim \int_{1/\modulus \WaringFareylevel \lesssim |t|} |z|^{-\frac{\dimension}{\degree}} \; dt \\
& \lesssim \absolutevalueof{\modulus \WaringFareylevel}^{1-\frac{\dimension}{\degree}} 
. 
\end{align*}
Then 
\begin{align*}
\sum_{\latticepoint\in \lattice} \twistedGausssum{\latticepoint} \Psi(\modulus \freqpoint-\latticepoint)\left( I_{\radius} - J_{\radius} \right)(\freqpoint) 
& \lessapprox \modulus^{-\frac{\dimension}{\degree}} (\WaringFareylevel \modulus)^{1-\frac{\dimension}{\degree}} \\
& = \modulus \, \WaringMajorarclevel^{\dimension-\degree+1-\frac{\dimension}{\degree}}  
. 
\end{align*}
\end{proof}

We need one more ingredient before we can prove The Approximation Formula. In particular, we need to understand the dyadic maximal operator 
$ \dyadicmaxHLop^{\Fareyfraction} 
:= \disup{\dyadicradius}{\radius} \absolutevalueof{\HLop^{\Fareyfraction}}$ 
for large $\modulus$. Hua's bound for Gauss sums extends to $\maxHLop^{\Fareyfraction}$ where $\maxHLop^{\Fareyfraction}$ is defined as $\sup_{\radius \in \acceptableradii} \absolutevalueof{\HLop^{\Fareyfraction} \fxn}$. The following bound for $\maxHLop$ clearly implies the same bound for the dyadic maximal operator $\dyadicmaxHLop^{\Fareyfraction}$. 
\begin{lemma}\label{lemma:maxHLop_arc_bound}
If $\dimension > \frac{\degree}{2}+1$ and $\frac{2\dimension-2+\degree}{2(\dimension-1)} < p \leq 2$, we have the bound 
\begin{equation}\label{max_HL_op_bound}
\lpnorm{p}{\maxHLop^{\Fareyfraction} \fxn} 
\lessapprox \modulus^{-\frac{\dimension}{\degree} \inparentheses{2-\frac{2}{p}}} \lpnorm{p}{\fxn} 
\end{equation}
\end{lemma}
\begin{proof}
We use the separation trick in \cite{MSW} to separate the study of $\maxHLop^{\Fareyfraction}$ into arithmetic and analytic parts. 
Let $\Psi_1$ be a smooth function such that $\Psi_1(\freqpoint) \cdot \Psi(\freqpoint) = \Psi(\freqpoint)$. We separate $\HLop^{\Fareyfraction} \fxn$ into an arithmetic factor and a continuous factor by writing
\begin{align*}
\latticeFT{\HLop^{\Fareyfraction}}(\freqpoint) 
&= \eof{\radius^\degree \Fareyfraction } \sum_{\latticepoint\in \lattice} \twistedGausssum{\latticepoint}\Psi(\modulus\freqpoint-\latticepoint) \Psi_1(\modulus\freqpoint-\latticepoint) \contFT{d\higherorderspheremeasure_{\radius}}(\freqpoint-\latticepoint/\modulus) \\
&= \inparentheses{ \eof{\radius^\degree \Fareyfraction} \sum_{\latticepoint\in \lattice} \twistedGausssum{\latticepoint}\Psi(\modulus\freqpoint-\latticepoint) } \inparentheses{  \sum_{\latticepoint\in \lattice} \Psi_1(\modulus\freqpoint-\latticepoint) \contFT{d\higherorderspheremeasure_{\radius}}(\freqpoint-\latticepoint/\modulus) } \\
&=: m_{\radius}^{\Fareyfraction}(\freqpoint) \cdot n_{\radius}^{\Fareyfraction}(\freqpoint)
\end{align*}
For each $\radius \in \acceptableradii$ our multiplier $\latticeFT{\HLop^{\Fareyfraction}}$ is the product of the two commuting multipliers, $m^{\Fareyfraction}_\radius$ and $n^{\Fareyfraction}_\radius$ with convolution operators say, $M^{\Fareyfraction}_\radius$ and $N^{\Fareyfraction}_\radius$ respectively. 
Then the $\ell^p(\lattice)$-norm of $\maxHLop^{\Fareyfraction}$ is bounded by the product of the $\ell^p(\lattice)$-norms of $M_{*}^{\Fareyfraction} := \sup_{\radius \in \acceptableradii} \absolutevalueof{M^{\Fareyfraction} \fxn}$ and $N_{*}^{\Fareyfraction} := \sup_{\radius \in \acceptableradii} \absolutevalueof{N^{\Fareyfraction} \fxn}$. 

To study the arithmetic part $M^{\Fareyfraction}_\radius$, we use Magyar--Stein--Wainger's $1/\modulus$-periodic inequality (Lemma~\ref{lemma:MSW_periodic_inequality}). In Lemma~\ref{lemma:MSW_periodic_inequality} we take $g(\latticepoint) = \eof{\radius^\degree \Fareyfraction} \twistedGausssum{\latticepoint}$, then $\latticeFT{g}(n) = 1$ for all $n \in \lattice$. Hua's bound for Gauss sums implies that 
\begin{equation}
\lpnorm{p}{M^{\Fareyfraction}_* \fxn} 
\lessapprox \modulus^{-\frac{\dimension}{\degree} \inparentheses{2-\frac{2}{p}}} \lpnorm{p}{\fxn} 
\end{equation}
for $1 \leq p \leq 2$. 
The analytic part $N^{\Fareyfraction}_*$ is easily handled by the Magyar--Stein--Wainger transference lemma and the Proposition~\ref{prop:cont_max_fxn_bound_for_higher_order_spheres}. 
Proposition~\ref{prop:cont_max_fxn_bound_for_higher_order_spheres} implies 
\begin{equation}\label{cont_bound}
\Lpnorm{p}{\euclideanspace}{N^{\Fareyfraction}_* \fxn} 
\lesssim \Lpnorm{p}{\euclideanspace}{\fxn} 
\end{equation}
for $\dimension > \frac{\degree}{2}+1$ and $p > \frac{2\dimension-2+\degree}{2(\dimension-1)}$. Taking the Banach space to be $B = \ell^{\infty}(\acceptableradii)$ in the Magyar--Stein--Wainger transference lemma, \eqref{tranference_principle} and \eqref{cont_bound} imply 
\begin{equation}
\lpnorm{p}{N^{\Fareyfraction}_* \fxn} 
\lesssim \lpnorm{p}{\fxn} 
\end{equation}
for $\dimension > \frac{\degree}{2}+1$ and $p > \frac{2\dimension-2+\degree}{2(\dimension-1)}$. 
Therefore \eqref{max_HL_op_bound} holds for $\dimension > \frac{\degree}{2}+1$ and $p > \frac{2\dimension-2+\degree}{2(\dimension-1)}$. 
\end{proof}
%

\begin{proof}[Proof of The Approximation Formula]
Analogous to completing the singular series in Waring's problem -- see Chapter 4 of \cite{Davenport} -- we complete the approximation by defining the convolution operator 
\begin{equation}\label{HLop_completion}
\HLop := \sum_{\modulus = 1}^{\infty} \sum_{\unit \in \unitsmod{\modulus}} \HLop^{\Fareyfraction} 
. 
\end{equation}
Finally, the error term in The Approximation Formula is defined as 
\begin{equation}
\error 
:= \radius^{\dimension-\degree} \inbrackets{\inparentheses{\avgop^{Major} - B^{Major}} + \inparentheses{B^{Major} - \HLop^{Major}} + \sum_{\modulus > \WaringMajorarclevel} \sum_{\unit \in \unitsmod{\modulus}} \HLop^{\Fareyfraction}  + \avgop^{minor}} 
. 
\end{equation}
We easily check that $\avgop = \HLop + \radius^{\degree-\dimension} \error$ for all $\radius \in \acceptableradii$. 
By \eqref{max_HL_op_bound} of Lemma~\ref{lemma:maxHLop_arc_bound}, we have the bound 
\begin{equation}\label{singular_series_error_bound}
\lpnorm{2}{\disup{\dyadicradius}{\radius} \sum_{\modulus > \WaringMajorarclevel} \sum_{\unit \in \unitsmod{\modulus}} \HLop^{\Fareyfraction}} 
\lessapprox \dyadicradius^{2-\dimension/\degree} 
. 
\end{equation} 
By equations \eqref{dyadic_major_arc_bound}, \eqref{singular_series_error_bound} and \eqref{minor_arcs_max_fxn_bound}, we have
\begin{equation}
\lpnorm{2}{\disup{\dyadicradius}{\radius}{\error}} 
\lessapprox \dyadicradius^{\dimension-\degree} \inparentheses{\dyadicradius^{\degree-\dimension \powerloss} + \dyadicradius^{\degree+2-\frac{\dimension}{\degree}} + \dyadicradius^{2-\frac{\dimension}{\degree}} + \dyadicradius^{3-\frac{\dimension}{\degree}}} \lpnorm{2}{\fxn} 
. 
\end{equation}
\eqref{error_bound} of The Approximation Formula follows with $\dyadicsavings = \min \inbraces{\dimension-\degree(\degree+2), \degree- \dimension \powerloss}$ and  $\dimension > \max{\inbraces{\degree(\degree+2), \degree/\powerloss}}$.
\end{proof}

\section{Proof of Theorem~\ref{thm:Vinogradov_max_thm}}\label{section:maximal_theorem_proof}
In this section we finally come to the proof of Theorem~\ref{thm:Vinogradov_max_thm}. Before doing so, we need to understand the maximal operator $\maxHLop$. The boundedness of $\maxHLop$ and the approximation lemma will imply Theorem~\ref{thm:Vinogradov_max_thm}. 
The boundedness of $\maxHLop$ is contained in the next lemma and is an immediate corollary of Lemma~\ref{lemma:maxHLop_arc_bound} since $\frac{2\dimension-2+\degree}{2(\dimension-1)} < \frac{\dimension}{\dimension-\degree}$ for $\dimension > \degree$. 
\begin{lemma}\label{maxHLop_lemma}
If $\dimension > \frac{\degree}{2} + 1$ and $p > \frac{\dimension}{\dimension-\degree}$, then 
\begin{equation}\label{maxHLop_bound}
\lpnorm{p}{\maxHLop \fxn} 
\lesssim \lpnorm{p}{\fxn} 
. 
\end{equation}
\end{lemma}
%

\begin{proof}[Proof of Theorem~\ref{thm:Vinogradov_max_thm}]
Summing over dyadic $\dyadicradius = 2^j$ for $j \in \N$, \eqref{error_bound} of the approximation formula implies that 
\begin{equation}\label{error_bound_on_ell2}
\lpnorm{2}{\sup_{\radius \in \acceptableradii} \absolutevalueof{\avgop \fxn - \HLop \fxn}} 
\lesssim \lpnorm{2}{\fxn} 
. 
\end{equation}
Combining \eqref{error_bound_on_ell2} with \eqref{maxHLop_bound} of Lemma~\ref{maxHLop_lemma}, we prove Theorem~\ref{thm:Vinogradov_max_thm} for $p=2$ and $\dimension > \max{\inbraces{\degree(\degree+2), \degree/\powerloss}}$. 
By Theorem~\ref{thm:dyadic}, the dyadic maximal operator $\dyadicmaxop$ is bounded on $\ell^p(\lattice)$ when $\dimension > \max{\inbraces{2\degree, \degree/\powerloss }}$ and  $p > \max{\inbraces{\frac{\dimension}{\dimension-\degree}, \frac{\dimension}{\dimension -\degree/2\powerloss}}}$. 
By Lemma~\ref{lemma:maxHLop_arc_bound}, the dyadic maximal operator $\dyadicmaxHLop$ is bounded on $\ell^p(\lattice)$ when $\dimension > \frac{\degree}{2}+1$ and $p > \frac{\dimension}{\dimension-\degree}$. These last two observations imply that 
\begin{equation}\label{dyadic_error_bound_on_ellp}
\lpnorm{p}{\disup{\dyadicradius}{\radius}{\absolutevalueof{\avgop \fxn - \HLop \fxn}} } 
\lesssim \lpnorm{p}{\fxn} 
\end{equation}
for $\dimension > \max{\inbraces{2\degree, \degree/\powerloss }}$ and  $\max{\inbraces{\frac{\dimension}{\dimension-\degree}, \frac{\dimension}{\dimension -\degree/2\powerloss}} } < p \leq 2$. 
Interpolating \eqref{error_bound} with \eqref{dyadic_error_bound_on_ellp}, we find that there exists some $\epsilon(p) > 0$ such that  
\[
\lpnorm{p}{\disup{\dyadicradius}{\radius}{\absolutevalueof{\avgop \fxn - \HLop \fxn}} } 
\lesssim \dyadicradius^{-\epsilon(p)} \lpnorm{p}{\fxn} 
\]
for $\dimension > \max{\inbraces{2\degree, \degree/\powerloss }}$ and  $\max{\inbraces{\frac{\dimension}{\dimension-\degree}, \frac{\dimension}{\dimension -\degree/2\powerloss}} } < p \leq 2$. 
Summing over dyadic $\dyadicradius = 2^j$, we deduce
\begin{equation}\label{error_bound_on_ellp}
\lpnorm{p}{\sup_{\radius \in \acceptableradii} \absolutevalueof{\avgop \fxn - \HLop \fxn}} 
\lesssim \lpnorm{p}{\fxn} 
\end{equation}
for $\dimension > \max{\inbraces{2\degree, \degree/\powerloss }}$ and  $\max{\inbraces{\frac{\dimension}{\dimension-\degree}, \frac{\dimension}{\dimension -\degree/2\powerloss}} } < p \leq 2$. 
\eqref{maxHLop_bound} of Lemma~\ref{maxHLop_lemma} and \eqref{error_bound_on_ellp} imply 
\[
\lpnorm{p}{\maxop \fxn} 
\lesssim \lpnorm{p}{\fxn} 
\]
for $\max{\inbraces{\frac{\dimension}{\dimension-\degree}, \frac{\dimension}{\dimension -\degree/2\powerloss}} } < p \leq 2$. 
Interpolation with the trivial $\ell^{\infty}$ bound finishes the theorem.
\end{proof}
%
\begin{rem}
Since Conjectures~\ref{conjecture:higher_degree_strong_type} and \ref{conjecture:higher_degree_weak_type} are very difficult (requiring at least a full solution to Waring's problem), a slightly more reasonable problem is to prove $\maxop$ is bounded on $\ell^2(\lattice)$ for $\dimension >2\degree$. The above proposition shows that this is true for $\maxHLop$. So the difficulties lie in improving the bounds for the error term in the approximation formula, especially in the minor arcs bounds of \eqref{minor_arcs_max_fxn_bound}. 
\end{rem}

\section{Ergodic arithmetic $\degree$-spherical averages}\label{section:ergodic}
%
This section follows the method in \cite{Magyar_ergodic} to prove the pointwise ergodic theorems for $\degree$-spherical averages. 
Our set-up is a measure preserving system $\mps$, i.e. $\measurespace$ is a probability space, $\measure$ is its probability measure and $\measurepreservingtransform = (\measurepreservingtransform_1, \dots, \measurepreservingtransform_\dimension)$ is a family of commuting, invertible measure preserving transformations. This induces a $\Z^\dimension$-action on $\measurespace$ by 
$\measurespacepoint \mapsto \measurepreservingtransform^\latticepoint \measurespacepoint 
:= \measurepreservingtransform_1^{\latticepoint_1}\cdots \measurepreservingtransform_{\dimension}^{\latticepoint_\dimension} \measurespacepoint$ for all $\latticepoint \in \lattice$. 
This action naturally extends to functions on $\measurespace$; e.g. for a measurable function $\mfxn$ defined on $\measurespace$, let $\measurepreservingtransform^\latticepoint \mfxn(\measurespacepoint) = \mfxn(\measurepreservingtransform^\latticepoint \measurespacepoint)$. 
For a function $\mfxn: \measurespace \to \C$, we use this action to define the \emph{$\degree$-spherical averages of radius $\radius$} as 
\[
\arithmeticergodicsphericalaverage \mfxn(x) 
:= \inverse{N_{\dimension,\degree}(\radius)} \sum_{\latticepoint \in \arithmetichigherordersphere} \mfxn(\measurepreservingtransform^\latticepoint x) 
\]
and its \emph{$\degree$-spherical maximal function} as 
\[
\arithmeticergodicsphericalmaximal \mfxn 
:= \sup_{\radius \in \acceptableradii} \absolutevalueof{\arithmeticergodicsphericalaverage \mfxn} 
. 
\]
By the Calder\'on transference principle -- see \cite{Bourgain_maximal_ergodic} or \cite{Magyar_ergodic} -- the maximal theorems on $\lattice$ imply maximal theorems for any measure preserving system. In particular Theorem~\ref{thm:Vinogradov_max_thm} implies: 
\begin{mpsmaximaltheorem}
If $\arithmetichigherorderspheremaximal$ is bounded on $\ell^p(\lattice)$, then $\arithmeticergodicsphericalmaximal$ is bounded on $L^p(\measurespace, \measure)$; that is, for any function $\mfxn \in L^p(\measurespace,\measure)$,
\[
\Lpnorm{p}{\measurespace, \measure}{ \arithmeticergodicsphericalmaximal \mfxn}
\lesssim \Lpnorm{p}{\measurespace, \measure}{\mfxn} 
. 
\]
\end{mpsmaximaltheorem}

For our $\degree$-spherical ergodic theorems, we not only need to assume that the family of measure preserving transformations is ergodic, but is also strongly ergodic. 
Recall that a commuting family of measure preserving transformations is ergodic if $\measurepreservingtransform_1 \mfxn = \cdots = \measurepreservingtransform_\dimension \mfxn) = \mfxn$ implies that $\mfxn$ is constant, or equivalently, the only $\measurepreservingtransform$-invariant measurable sets have either measure 0 or 1. 
A family of measure preserving transformations is \emph{strongly ergodic} if for all $q \in \N$, the family $\measurepreservingtransform^q = (\measurepreservingtransform_1^q, \dots, \measurepreservingtransform_\dimension^q)$ 
is ergodic. 
This condition rules out certain atomic counterexamples -- see \cite{Magyar_ergodic} for more information -- and is necessary for our mean and pointwise ergodic theorems. In particular we will prove: 
\begin{Ltwoergodictheorem}\label{Ltwo_ergodic}
If hypothesis $\hypothesis{\powerloss}$ is true for some $0<\powerloss<1$
and $\measurepreservingtransform = (\measurepreservingtransform_1, \dots, \measurepreservingtransform_\dimension)$ is a commuting, strongly ergodic family of operators, then
\[
\Lpnorm{2}{\measurespace, \measure}{\lim_{\radius \to \infty} \arithmeticergodicsphericalaverage \mfxn - \int_\measurespace \mfxn d\measure} = 0 
\]
for $\dimension > \max{\{ \degree(\degree+2), \degree/\powerloss\}}$. 
\end{Ltwoergodictheorem}

A standard approximation argument using the maximal theorem for measure preserving systems and that $L^2$ is dense in $L^p$ allows us to immediately deduce convergence in $L^p$ for certain $p<2$.

\begin{cor}
If hypothesis $\hypothesis{\powerloss}$ is true for some $0<\powerloss<1$
and $\measurepreservingtransform = (\measurepreservingtransform_1, \dots, \measurepreservingtransform_\dimension)$ is a commuting, strongly ergodic family of operators, then we have convergence in $L^p(\measurespace, \measure)$ for $\dimension > \max{\{ \degree(\degree+2), \degree/\powerloss\}}$ and $p> \max{ \{ \frac{\dimension}{\dimension -\degree/2\powerloss}, \frac{\dimension}{\dimension-\degree} \}}$. 
\end{cor}

\begin{pointwiseergodictheorem}
If hypothesis $\hypothesis{\powerloss}$ is true for some $0<\powerloss<1$, $\measurepreservingtransform = (\measurepreservingtransform_1, \dots, \measurepreservingtransform_\dimension)$ is a family of strongly ergodic operators with $\dimension>\max{\{ \degree(\degree+2), \degree/\powerloss\}}$ and $\mfxn \in L^p(\measurespace,\measure)$ for some $p>\max{ \{ \frac{\dimension}{\dimension -\degree/2\powerloss}, \frac{\dimension}{\dimension-\degree} \}}$, then we have pointwise convergence a.e., that is
\[
\lim_{\radius \to \infty} \arithmeticergodicsphericalaverage \mfxn(x) = \int_\measurespace \mfxn d\measure
\]
for $\measure$-a.e. $x\in \measurespace$.
\end{pointwiseergodictheorem}

In particular, using the Weyl bound and Wooley's sup bound, the mean and pointwise ergodic theorems are true for
\begin{itemize}
\item $p>\frac{\dimension}{\dimension - \degree 2^{\degree-2}}$ and $\dimension>\degree 2^{\degree-1}$ if $2\leq \degree \leq 7$
\item $p>\frac{\dimension}{\dimension - \degree^2(\degree-2)}$ and $\dimension > 2\degree^2 (\degree-2)$ if $\degree \geq 8$.
\end{itemize}

\subsection{The Spectral Theorem for Unitary Operators}
Before turning to the proofs of our $\degree$-spherical $L^2$ and pointwise ergodic theorems, we recall a few facts from \cite{Magyar_ergodic} concerning the spectral theory of unitary operators. Let $\innerproductof{f}{g}:= \int_{\measurespace} f \conjugate{g} d\measure$ be the inner product of $f$ and $g$ in $L^2(\measurespace)$.
\begin{spectraltheorem}
For a commuting family $\measurepreservingtransform = (\measurepreservingtransform_1, \dots, \measurepreservingtransform_\dimension)$ of measure preserving transformations and a measurable $\mfxn : \measurespace \to \C$, there exists a Borel measure $\spectralmeasure_{\mfxn}$ on $\torus$ such that for any polynomial $P \in \Z[x_1, \dots, x_\dimension]$, 
\[
\innerproductof{P(\measurepreservingtransform_1, \dots, \measurepreservingtransform_\dimension)\mfxn}{\mfxn} = \int_{\torus} P(e(\spectralpoint)) \, d\spectralmeasure_{\mfxn}(\spectralpoint) .
\]

\end{spectraltheorem}
The spectral measure satisfies the following two properties.
\begin{enumerate}
\item Let $\spectralpoint \in \torus$. $\spectralmeasure_{\mfxn}(\spectralpoint) >0$ if and only if there exists $g \in L^2$ such that $\measurepreservingtransform_j g = e(\spectralpoint_j) g$ for all $j$. We say that $\spectralpoint$ is a joint eigenvalue of $\measurepreservingtransform$, and $g$ is a joint eigenfunction of $\measurepreservingtransform$. 
\item If $\measurepreservingtransform$ is ergodic, then $\spectralmeasure_{\mfxn}(0) = \absolutevalueof{\int_{\measurespace} \mfxn d\measure}^2$. 
\end{enumerate}

Using (1), we have the following characterization of strongly ergodic in terms of the spectrum -- see \cite{Magyar_ergodic} for more details. 
\begin{fact}
A commuting family $\measurepreservingtransform$ of tranformations is strongly ergodic if and only if $\spectralmeasure_{\mfxn}(\Q^\dimension \setminus \{0\}) \equiv 0$ for any measurable function $\mfxn$. 
\end{fact}
%

%
\subsection{Proof of the $L^2$ Ergodic Theorem}
Without loss of generality, assume that $\int_{\measurespace} \mfxn \, d\measure=0$ so that $\spectralmeasure_{\mfxn}(0)=0$. This and the strong ergodicity condition imply that $\spectralmeasure_{\mfxn}(\Q^\dimension) \equiv 0$. Using the spectral theorem, we relate this to $\higherorderspheremultiplier$:
\begin{align*}
\Lpnorm{2}{\measurespace}{\arithmeticergodicsphericalaverage \mfxn - \int_{\measurespace} \mfxn \, d\measure}^2
& = \Lpnorm{2}{\measurespace}{\arithmeticergodicsphericalaverage \mfxn}^2 \\
& = \int_{\torus} |\higherorderspheremultiplier(\freqpoint)|^2 \, d\spectralmeasure_{\mfxn}(\freqpoint) 
. 
\end{align*}
The $L^2$ ergodic theorem will follow from Lebesgue's dominated convergence theorem and the strong ergodicity condition once we prove the following lemma. 
\begin{lemma} For each irrational $\toruspoint$ e.g. $\toruspoint \not\in \Q^\dimension \cap \unitcube$,
\[
\lim_{\radius \to \infty} \higherorderspheremultiplier(\toruspoint) = 0 
. 
\]
\end{lemma}

\begin{proof}
By The Approximation Formula, 
$\higherorderspheremultiplier(\toruspoint) 
= \arithmeticFT{\HLop}(\toruspoint) + \radius^{\degree-\dimension} \arithmeticFT{\multerror}(\toruspoint) 
= \radius^{\degree-\dimension} \inparentheses{\sum_{\modulus=1}^{\infty} \HLmultiplier_{\radius}^{\unit/\modulus}(\toruspoint) + \arithmeticFT{\multerror}(\toruspoint)}$. 
The dyadic bound \eqref{error_bound} for $\multerror$ implies $\arithmeticFT{\multerror}(\toruspoint) \to 0$ uniformly for $\xi \in \torus$ as $\radius \to \infty$. 
We are left to consider the main term $\radius^{\degree-\dimension} \sum_{\modulus=1}^{\infty} \HLmultiplier_{\radius}^{\unit/\modulus}(\toruspoint)$. 
Let $0<\epsilon \ll 1$, then Hua's bound for Gauss sums implies that there exists a $\bigmodulus(\epsilon) \in \N$ such that $\sum_{\modulus \geq \bigmodulus} \sum_{\unit \in \unitsmod{\modulus}} \HLmultiplier_{\radius}^{\unit/\modulus}(\toruspoint) \lesssim \epsilon$. So we are reduced to studying finitely many multipliers $\HLmultiplier_{\radius}^{\unit/\modulus}$ for $\modulus < \bigmodulus(\epsilon)$ and $\unit \in \unitsmod{\modulus}$. For a fixed irrational $\xi \in \torus$, choose $\delta(\bigmodulus)$ positive and sufficiently small so that $|\toruspoint-\latticepoint/\modulus| > \delta$ for all $\modulus \leq \bigmodulus$ and $\latticepoint \in \lattice$. 
The estimate \eqref{kspherical_decay} of the Bruna--Nagel--Wainger theorem implies that for $\radius > \radius(\delta)$ sufficiently large depending on $\delta$, we have $\contFT{d\higherorderspheremeasure_{\radius}}(\freqpoint) \lesssim \left( \frac{\delta}{\radius} \right)^{\frac{\dimension-1}{\degree}}$. Taking $\radius$ sufficiently large we make $\contFT{d\higherorderspheremeasure_{\radius}}(\freqpoint)$ arbitrarily small. The Magyar--Stein--Wainger transference lemma implies that $\HLmultiplier_{\radius}^{\unit/\modulus}(\toruspoint) \to 0$ as $\radius \to \infty$ for each irrational $\toruspoint$.
\end{proof}

%
\subsection{Proof of the Pointwise Ergodic Theorem}
Our proof of the pointwise ergodic theorem for arithmetic $\degree$-spherical averages is more intricate than our proof of the $L^2$ mean ergodic theorem because there is no natural basis for which we can prove pointwise convergence. This is the same obstacle that Bourgain over came in his study of square averages in \cite{Bourgain_pointwise_ergodic} and \cite{Bourgain_arithmetic_sets}; in these papers Bourgain introduced proved an oscillation inequality to deduce pointwise convergence. This method was exploited in \cite{Magyar_ergodic}; our proof relies on proving the same oscillation inequality. 

\subsubsection{Reduction to the oscillation inequality}
We commence with a few reductions. Without loss of generality, we assume that our function has mean 0; that is, $\int_{\measurespace} \mfxn d\measure =0$. By the maximal theorem for measure preserving systems, we reduce to proving the theorem on a dense subset of $L^p$; we choose $L^\infty$. We want to show for any bounded function $\mfxn$ of mean 0 in $L^2(\measurespace)$ that 
\[
\measure \left(\{ \measurespacepoint \in \measurespace: \limsup_{\radius \to \infty} | \arithmeticergodicsphericalaverage \mfxn(\measurespacepoint)| > 0 \} \right)
= 0 .
\]
Assume not for contradiction; thus, we can assume that for positive $\alpha$,
\begin{equation}\label{positive_oscillation}
\measure \left(\{ \measurespacepoint \in \measurespace: \limsup_{\radius \to \infty} | \arithmeticergodicsphericalaverage \mfxn(\measurespacepoint)| > 2\alpha \} \right)
> 2\alpha .
\end{equation}
\eqref{positive_oscillation} implies that for any radius $\radius_1 \in \acceptableradii$, we can find another radius $\radius_2 \in \acceptableradii$ such that
\[
\measure \left(\{ \measurespacepoint \in \measurespace: \sup_{\radius_1 \leq \radius < \radius_2} | \arithmeticergodicsphericalaverage \mfxn(\measurespacepoint)| > \alpha \} \right)
> \alpha .
\]
We use this process to build a sequence of radii $\radius_j \in \acceptableradii$ for $j\in \N$ and truncated maximal functions 
\[
\truncatedarithmeticergodicsphericalmaximal{j} \mfxn := \sup_{\radius_j \leq \radius < \radius_{j+1} } \arithmeticergodicsphericalaverage \absolutevalueof{\mfxn} 
.
\]
We will show:
\begin{oscillation_inequality_mps}
For all $\epsilon>0$, there is a $J = J(\epsilon)$ sufficiently large such that for all $j>J$, 
\begin{equation}\label{oscillation_inequality_for_mps}
\Lpnorm{2}{\measurespace}{\truncatedarithmeticergodicsphericalmaximal{j} \mfxn } 
\lesssim \epsilon \Lpnorm{2}{\measurespace}{\mfxn} .
\end{equation}
\end{oscillation_inequality_mps}
Choosing $\epsilon$ small enough in the oscillation inequality for measure preserving systems, we contradict \eqref{positive_oscillation} by Chebychev's inequality for $L^2(\measurespace)$. Therefore the pointwise ergodic theorem for $\degree$-spherical maximal functions is proved if the oscillation inequality for measure preserving systems is proved. 

We prove \eqref{oscillation_inequality_for_mps} by an appeal to the Calder\'on transference principle; this leads to an oscillation inequality on the product space $\lattice \times \measurespace$. 
We lift $\mfxn$ to the \emph{transfer function} $\transferfxn(\measurespacepoint, \latticepoint) = \mfxn(\measurepreservingtransform^{\latticepoint} \measurespacepoint)$ for all $\measurespacepoint \in \measurespace$ and $\latticepoint \in \lattice$. We truncate $\transferfxn$ to lattice points $\latticepoint$ in a large cube $\discretecube{N} := \inbraces{n \in \lattice : \absolutevalueof{n_i} \leq N \text{ for } i=1, \dots, \dimension}$ with $N \gg \radius_j$; that is, redefine $\transferfxn$ to be $\transferfxn \cdot \indicator{\discretecube{N}}$. Define the \emph{lifted $\degree$-spherical averages}
\[
\arithmetichigherordersphereaverage \transferfxn(\measurespacepoint, n)
:= \inverse{\numberoflatticepoints(\radius)} \sum_{\latticepoint \in \arithmetichigherordersphere} \measurepreservingtransform^\latticepoint \mfxn(\measurepreservingtransform^n \measurespacepoint)
\]
for $n \in \lattice$, and the \emph{lifted $\degree$-spherical maximal function} $\truncatedarithmetichigherorderspheremaximal{j} \transferfxn(\measurespacepoint, \cdot)$ similarly to $\truncatedarithmeticergodicsphericalmaximal{j} \mfxn(\measurespacepoint)$. 
Give $\measurespace \times \lattice$ the product measure of $\measure$ on $\measurespace$ and the counting measure on $\lattice$. We transfer the oscillation inequality \eqref{oscillation_inequality_for_mps} to $\measurespace \times \lattice$. 
\begin{oscillation_inequality_lattice}
For all $\epsilon>0$, there exists a $J(\epsilon)$ such that for $j>J$, we have
\begin{equation}\label{oscillation_inequality_for_lattice}
\Lpnorm{2}{\measurespace \times \lattice}{\truncatedarithmetichigherorderspheremaximal{j} \transferfxn } 
\lesssim \epsilon \Lpnorm{2}{\measurespace \times \lattice}{\transferfxn} .
\end{equation}
\end{oscillation_inequality_lattice}
\begin{proof}[Proof of the Oscillation Inequality for Measure Preserving Systems]
Once the transfer oscillation inequality \eqref{oscillation_inequality_for_lattice} is proved, we have
\[
\Lpnorm{2}{\measurespace \times \lattice}{\transferfxn} 
= \sizeof{\discretecube{N}} \Lpnorm{2}{\measurespace}{\mfxn}
\]
while
\[
\sizeof{\discretecube{N-\radius_j}} \Lpnorm{2}{\measurespace}{\truncatedarithmeticergodicsphericalmaximal{j} \mfxn }
\leq \Lpnorm{2}{\measurespace \times \lattice}{\truncatedarithmetichigherorderspheremaximal{j} \transferfxn }
.
\]
Therefore, 
\[
\frac{\sizeof{\discretecube{N-\radius_j}}}{\sizeof{\discretecube{N}}} \Lpnorm{2}{\measurespace}{\truncatedarithmeticergodicsphericalmaximal{j} \mfxn }
\lesssim \epsilon \Lpnorm{2}{\measurespace}{\mfxn}
.
\]
Taking $N \to \infty$, we prove \eqref{oscillation_inequality_for_mps}.
\end{proof}

Our final step is to prove the oscillation inequality \eqref{oscillation_inequality_for_lattice}.
%
\subsubsection{Proof of the Transfer Oscillation Inequality}
%
First we extend The Approximation Formula to the lifted averages. Define the partial Fourier transform in the variable $\latticepoint \in \lattice$ by $\arithmeticFT{\transferfxn}(\measurespacepoint,\toruspoint) := \sum_{\latticepoint \in \lattice} \transferfxn(\measurespacepoint,\latticepoint) \eof{\latticepoint \cdot \freqpoint}$ for $\freqpoint \in \torus$. This allows us to extend $\eqref{HL_multiplier_defn}$ to $\measurespace \times \lattice$. We define the convolution operators $\transferHLoponarc$ by the multipliers 
\[
\arithmeticFT{\transferHLoponarc \transferfxn}(\measurespacepoint, \freqpoint) 
:= \HLmultiplier^{\Fareyfraction}_{\radius}(\freqpoint) \arithmeticFT{\transferfxn}(\measurespacepoint, \freqpoint) 
. 
\] 
Similary we extend the definition of the convolution operators $\HLop$, $B_{\radius}^{\Fareyfraction}$, $B_{\radius}$, $\error$, $M^{\Fareyfraction}_{\radius}$ and $N^{\Fareyfraction}_{\radius}$. Finally we define their truncated maximal functions similar to 
$\truncatedarithmetichigherorderspheremaximal{j} \transferfxn$. 
We repeat some of the previous analysis. 
Hua's bound for Gauss sums extends the bound \eqref{max_HL_op_bound}: for all $\modulus \in \N$ and $\unit \in \unitsmod{\modulus}$, we have the bound 
\[
\Lpnorm{2}{\measurespace \times \lattice}{\transfermaxHLoponarc \transferfxn} 
\lessapprox \modulus^{-\dimension/\degree} \Lpnorm{2}{\measurespace \times \lattice}{\transferfxn} 
. 
\] 
If
$\epsilon > 0$, then there exists $\bigmodulus(\epsilon)$ sufficiently large such that 
\[
\Lpnorm{2}{\measurespace \times \lattice}{\sum_{\modulus \geq \bigmodulus(\epsilon)} \sum_{\unit \in \unitsmod{\modulus}} \transfermaxHLoponarc \transferfxn} 
\lesssim \epsilon \Lpnorm{2}{\measurespace \times \lattice}{\transferfxn} 
. 
\]
Thus we are reduced to proving for all $\modulus < \bigmodulus(\epsilon)$ and $\unit \in \unitsmod{\modulus}$, 
\[
\Lpnorm{2}{\measurespace \times \lattice}{\transfermaxHLoponarc \transferfxn} 
\lessapprox \epsilon \modulus^{-\dimension/\degree} \Lpnorm{2}{\measurespace \times \lattice}{\transferfxn}
\]
with implicit constants independent of $\unit$ and $\modulus$. 
Recall that we decomposed $\HLmultiplier_{\radius}^{\unit/\modulus} = m_{\radius}^{\Fareyfraction} \cdot n_{\radius}^{\Fareyfraction}$. 
For the remainder of the argument we suppress the dependence of $\Fareyfraction$ in $n_{\radius}^{\Fareyfraction}$ and simply write $n_{\radius}^{\Fareyfraction} = n_{\radius}$. 
Similar to the analysis in the proof of Theorem~\ref{thm:Vinogradov_max_thm}, we reduce to studying the convolution operator $\contoperator_{\radius}$ with multiplier 
\[
\contmultiplier_{\radius}(\freqpoint)
 = \sum_{\latticepoint \in \lattice} \Psi_1(\modulus \freqpoint - \latticepoint) \contFT{d\surfacemeasure_{\radius}}({\freqpoint-\latticepoint/\modulus})
\] 
where $\Psi_1$ is a smooth function supported in $[-1/2,1/2]^{\dimension}$ such that $\Psi_1 \Psi = \Psi$, and its truncated maximal functions $N^{\Fareyfraction, *}_j$. 
$m_{\radius}^{\Fareyfraction}$ is handled by Magyar--Stein--Wainger's $1/\modulus$-periodic inequality (Lemma~\ref{lemma:MSW_periodic_inequality}) and Hua's bound for Gauss sums and as before allows us to sum up the following bounds for $n_{\radius}^{\Fareyfraction}$. 
Note that $n_{\radius}^{\Fareyfraction}$ does not depend on $\unit \in \unitsmod{\modulus}$, and that for each $\xi \in \torus$ and $\modulus < \bigmodulus$, there is a unique $\latticepoint \in \lattice$ such that $\Psi_1(\modulus \xi - \latticepoint) \not= 0$. 
We work out the case where $\latticepoint/\modulus = 0$ using the mean 0 condition. The cases where $\latticepoint/\modulus \not=0$ are similar, but use the strong ergodicity condition in place of the mean zero condition.

For the sequence $\radius_j$, let $\bumpfxn_j$ be a bump function supported in $\inbraces{\norm{\toruspoint}_{\infty} \leq \radius_j^{-1}}$, then
\[
\contmultiplier_{\radius} 
= \bumpfxn_j \contmultiplier_{\radius} + (1-\bumpfxn_j) \contmultiplier_{\radius}
=: \contmultiplier_{\radius}^{j,low} + \contmultiplier_{\radius}^{j,high}
\]
is a decomposition of $\contmultiplier_{\radius}$ into \emph{low} and \emph{high} frequencies respectively. $\contmultiplier_{\radius}^{j,high}$ is supported in the annulus $\inbraces{\radius_j^{-1} < \norm{\toruspoint}_{\infty} \leq 1/2\modulus}$ while $\contmultiplier_{\radius}^{j,low}$ is supported in the ball $\inbraces{\norm{\toruspoint}_{\infty} \leq \radius_j^{-1}}$. Notice that our cut-off $\radius_{j}^{-1}$ does not depend on $\modulus \leq \bigmodulus$ and since $\bigmodulus$ only depends on $\epsilon$, we can assume that $\radius_j$ is large enough so that $\inbraces{\radius_j^{-1} < \norm{\toruspoint}_{\infty} \leq 1/2\modulus}$ is non-empty. 
Our analysis now splits into proving bounds for the low and high frequency parts. The next proposition handles the high frequency part. 
\begin{prop}
If $\epsilon>0$, then there exists a $J^{high}(\epsilon)$ sufficiently large such that for all $\modulus \in \N$, $\unit \in \unitsmod{\modulus}$ and $j > J$, 
\begin{equation}\label{high_frequency_oscillation}
\Lpnorm{2}{\measurespace \times \lattice}{\contoperator^{*,high}_{j}\transferfxn}
\lesssim \epsilon \Lpnorm{2}{\measurespace \times \lattice}{\transferfxn} 
\end{equation}
with implicit constants independent of $\unit$ and $\modulus$. 
\end{prop}
%
\begin{proof}
We now show that the high frequency part is equiconvergent with the low frequency part.
Let
\[
\contoperator_{j}^{*,high} \fxn 
= \sup_{\radius > \radius_j} \absolutevalueof{\contoperator_{\radius}^{high} \fxn} 
.
\]
By \eqref{kspherical_decay}, Proposition~\ref{prop:cont_max_fxn_bound_for_higher_order_spheres} and Calder\'on transference principle, we have the $L^2(\measurespace \times \lattice)$ estimates
\[
\Lpnorm{2}{\measurespace \times \lattice}{\contoperator_{j}^{*,high} \transferfxn} 
\lesssim \radius_j^{-\frac{\dimension-1}{\degree}} \Lpnorm{2}{\measurespace \times \lattice}{\transferfxn}
\]
with implicit constants that depend only on the degree $\degree$ and dimension $\dimension$. 
Choosing $J^{high}$ sufficiently large we conclude the proposition. 
\end{proof}

Our final proposition handles the low frequency part and is the only place where the mean zero and strong ergodicity conditions are used. 
\begin{prop}
If $\epsilon>0$ and $\bigmodulus \in \N$, then there exists a $J^{low}(\epsilon, \bigmodulus)$ sufficiently large such that for all $\modulus \leq \bigmodulus$, $\unit \in \unitsmod{\modulus}$ and $j > J^{low}(\epsilon, \bigmodulus)$, 
\begin{equation}\label{low_frequency_oscillation}
\Lpnorm{2}{\measurespace \times \lattice}{\contoperator^{*,low}_{j}\transferfxn}
\lesssim \epsilon \Lpnorm{2}{\measurespace \times \lattice}{\transferfxn} 
\end{equation}
with implicit constants independent of $\unit$ and $\modulus$. 
\end{prop}
%
\begin{proof}
We linearize the maximal function $\contoperator^{*,low}_{j}\transferfxn$ to $\contoperator_{\radius(\measurespacepoint)}^{low}\transferfxn(\measurespacepoint,\latticepoint)$ where for each $\measurespacepoint \in \measurespace$, there exists $\radius_j < \radius(\measurespacepoint) \in \acceptableradii$ such that $\contoperator^{*,low}_{j}\transferfxn(\measurespacepoint,\latticepoint) = \contoperator_{\radius(\measurespacepoint)}^{low}\transferfxn(\measurespacepoint,\latticepoint)$. 
Plancherel's theorem implies 
\[
\Lpnorm{2}{\measurespace \times \lattice}{\contoperator^{*,low}_{j}\transferfxn(\measurespacepoint,\latticepoint)}^2 
= \int_{\torus}
\absolutevalueof{
\contmultiplier_{\radius(\measurespacepoint)}^{low} \cdot \arithmeticFT{\transferfxn}(\measurespacepoint,\freqpoint) \, d\freqpoint
}^2 
. 
\]
Therefore we have the identities
\[
\absolutevalueof{\arithmeticFT{\transferfxn}(\measurespacepoint,\freqpoint)}^2 = \sum_{m,n} \transferfxn(\measurespacepoint,m) \overline{\transferfxn(\measurespacepoint,n)} e((m-n)\xi)
,
\]
and
\begin{align*}
\int_{\measurespace} \absolutevalueof{\arithmeticFT{\transferfxn}(\measurespacepoint, \freqpoint)}^2 \, d\measure(\measurespacepoint)
& = \int_{\measurespace} \sum_{m,n} \transferfxn(\measurespacepoint,m) \overline{\transferfxn(\measurespacepoint,n)} e((m-n)\freqpoint) \, d\measure(\measurespacepoint) \\
& = \sum_{m,n} e((m-n)\freqpoint) \int_{\measurespace} \transferfxn(\measurespacepoint,m) \overline{\transferfxn(\measurespacepoint,n)} \, d\measure(\measurespacepoint) \\
& = \sum_{m,n} \innerproductof{T^{m-n} \fxn}{\fxn} \eof{(m-n) \freqpoint} 
. 
\end{align*}
By the Spectral theorem,
\[
\innerproductof{T^{m-n} \fxn}{\fxn} = \int_{\torus} \eof{(m-n)\spectralpoint} \, d\spectralmeasure_{\fxn}(\spectralpoint)
,
\]
so that
\begin{align*}
& \int_{\torus} \absolutevalueof{\contmultiplier_{\radius(\measurespacepoint)}^{low}(\freqpoint)}^2 
\sum_{m,n \in \lattice} \innerproductof{T^{m-n} \fxn}{\fxn} \eof{(m-n) \freqpoint} \, d\freqpoint \\
& = \int_{\torus} \absolutevalueof{\contmultiplier_{\radius(\measurespacepoint)}^{low}(\freqpoint)}^2 \left[ \int_{\torus} 
\sum_{m,n \in \discretecube{N}} \eof{(m-n)\spectralpoint} \eof{(m-n) \freqpoint} \, d\spectralmeasure_{\fxn}(\spectralpoint)
\right] \, d\freqpoint 
. 
\end{align*}
By Fubini's theorem and orthogonality, this is just 
\begin{equation}\label{unnormalized_spectral_equality}
\int_{\torus} \absolutevalueof{\contmultiplier_{\radius(\measurespacepoint)}^{low}(\freqpoint)}^2 
\int_{\torus} \sum_{n} \#\inbraces{(m_1,m_2) \in \discretecube{N} \times \discretecube{N} : m_1-m_2 = n} \cdot e(n\spectralpoint) e(n \freqpoint) \, d\spectralmeasure_{\fxn}(\spectralpoint) \, d\freqpoint . 
\end{equation}

If we normalize the sum in \eqref{unnormalized_spectral_equality} to define the sequence
\[
a_N(n) :=  \frac{\#\{(m_1,m_2) \in \discretecube{N} \times \discretecube{N} : m_1-m_2 = n \}} {\sizeof{\discretecube{N}}} ,
\]
then $a_N \to 1$ as $N \to \infty$ which means that $\arithmeticFT{a_N}$ is a \emph{delta sequence}; that is, $\arithmeticFT{a_N} \to \delta$ as $N \to \infty$. Therefore, 
\begin{align*}
\sizeof{\discretecube{N}}^{-1} \Lpnorm{2}{\measurespace \times \lattice}{\contoperator_{j,low}^{*}\transferfxn(\measurespacepoint,\latticepoint)}^2 
& = \int_{\torus}
\absolutevalueof{
(\contmultiplier_{\radius(\measurespacepoint)}^{low})^2 * \arithmeticFT{a_N}
(\spectralpoint)}
d\spectralmeasure_{\fxn}(\spectralpoint) 
\end{align*}
Now we make use of the fact that multiplier is localized to low frequencies. For all $\epsilon>0$, there exists $M(\epsilon) \in \N$ such that $\absolutevalueof{\arithmeticFT{a_N} - \delta} < \epsilon$ for all $N>M$ and 
\begin{align*}
\int_{\torus}
\absolutevalueof{ \absolutevalueof{
\contmultiplier_{\radius(\measurespacepoint)}^{low}}^2 * \arithmeticFT{a_N}
(\spectralpoint) \;
d\spectralmeasure_{\fxn}(\spectralpoint) }
& \lesssim \spectralmeasure_{\fxn}(\absolutevalueof{\spectralpoint} \leq \radius_j^{-1})+ \epsilon \spectralmeasure_{\fxn}(\torus) \\
& = \spectralmeasure_{\fxn}(\absolutevalueof{\spectralpoint} \leq \radius_j^{-1}) + \epsilon \Lpnorm{2}{\measurespace}{\fxn}^2 
. 
\end{align*}
As $j \to \infty$, $\spectralmeasure_{\fxn}(\absolutevalueof{\spectralpoint} \leq \radius_j^{-1}) \to \spectralmeasure_{\fxn}(0)$, but $\spectralmeasure_{\fxn}(0) = \int_{\measurespace} \fxn d \measure = 0$. Choosing $j$ large enough, we are finished.
\end{proof}
%

\section*{Acknowledgements}
The author would like to thank Elias Stein for introducing the author to this problem, and for sharing his insights and intuitions while providing generous encouragement. 

\bibliographystyle{amsalpha}
\bibliography{references}

\providecommand{\bysame}{\leavevmode\hbox to3em{\hrulefill}\thinspace}
\providecommand{\MR}{\relax\ifhmode\unskip\space\fi MR }
\providecommand{\MRhref}[2]{%
  \href{http://www.ams.org/mathscinet-getitem?mr=#1}{#2}
}
\providecommand{\href}[2]{#2}
\begin{thebibliography}{MSW02}

\bibitem[AS06]{AS}
M.~Avdispahi{\'c} and L.~Smajlovi{\'c}, \emph{On maximal operators on
  {$k$}-spheres in {$\Bbb Z^n$}}, Proc. Amer. Math. Soc. \textbf{134} (2006),
  no.~7, 2125--2130 (electronic). \MR{2215783 (2007a:42037)}

\bibitem[Bou85]{Bourgain}
J.~Bourgain, \emph{Estimations de certaines fonctions maximales}, C. R. Acad.
  Sci. Paris S\'er. I Math. \textbf{301} (1985), no.~10, 499--502. \MR{812567
  (87b:42023)}

\bibitem[Bou88a]{Bourgain_maximal_ergodic}
\bysame, \emph{On the maximal ergodic theorem for certain subsets of the
  integers}, Israel J. Math. \textbf{61} (1988), no.~1, 39--72. \MR{937581
  (89f:28037a)}

\bibitem[Bou88b]{Bourgain_pointwise_ergodic}
\bysame, \emph{On the pointwise ergodic theorem on {$L^p$} for arithmetic
  sets}, Israel J. Math. \textbf{61} (1988), no.~1, 73--84. \MR{937582
  (89f:28037b)}

\bibitem[Bou89]{Bourgain_arithmetic_sets}
\bysame, \emph{Pointwise ergodic theorems for arithmetic sets}, Inst. Hautes
  \'Etudes Sci. Publ. Math. (1989), no.~69, 5--45, With an appendix by the
  author, Harry Furstenberg, Yitzhak Katznelson and Donald S. Ornstein.
  \MR{1019960 (90k:28030)}

\bibitem[CM86]{CowlingMauceri}
M.~Cowling and G.~Mauceri, \emph{Inequalities for some maximal functions.
  {II}}, Trans. Amer. Math. Soc. \textbf{296} (1986), no.~1, 341--365.
  \MR{837816 (87m:42013)}

\bibitem[Dav05]{Davenport}
H.~Davenport, \emph{Analytic methods for {D}iophantine equations and
  {D}iophantine inequalities}, second ed., Cambridge Mathematical Library,
  Cambridge University Press, Cambridge, 2005, With a foreword by R. C.
  Vaughan, D. R. Heath-Brown and D. E. Freeman, Edited and prepared for
  publication by T. D. Browning. \MR{2152164 (2006a:11129)}

\bibitem[Har79]{HL}
G.~H. Hardy, \emph{Collected papers of {G}. {H}. {H}ardy. {V}ol. {I-VII}}, The
  Clarendon Press Oxford University Press, New York, 1979, Including joint
  papers with J. E. Littlewood and others, Edited by L. S. Bosanquet, I. W.
  Busbridge, Mary L. Cartwright, E. F. Collingwood, H. Davenport, T. M. Flett,
  H. Heilbronn, A. E. Ingham, R. Rado, R. A. Rankin, W. W. Rogosinski, F.
  Smithies, E. C. Titchmarsh and E. M. Wright. \MR{527275 (81e:01028)}

\bibitem[HL11]{Hu_Li1}
Y.~{Hu} and X.~{Li}, \emph{Discrete fourier restriction associated with
  schrodinger equations}, ArXiv e-prints (2011), 1--19.

\bibitem[HW98]{HardyWright}
G.~H. Hardy and E.~M. Wright, \emph{An introduction to the theory of numbers},
  fifth ed., Oxford University Press, Oxford, 1998.

\bibitem[IKM10]{IKM}
I.~A. Ikromov, M.~Kempe, and D.~M{\"u}ller, \emph{Estimates for maximal
  functions associated with hypersurfaces in {$\Bbb R^3$} and related problems
  of harmonic analysis}, Acta Math. \textbf{204} (2010), no.~2, 151--271.
  \MR{2653054 (2011i:42026)}

\bibitem[Ion04]{Ionescu}
A.~D. Ionescu, \emph{An endpoint estimate for the discrete spherical maximal
  function}, Proc. Amer. Math. Soc. \textbf{132} (2004), no.~5, 1411--1417
  (electronic).

\bibitem[Mag97]{Magyar_dyadic}
A.~Magyar, \emph{{$L^p$}-bounds for spherical maximal operators on {$\bold
  Z^n$}}, Rev. Mat. Iberoamericana \textbf{13} (1997), no.~2, 307--317.
  \MR{1617657 (99d:42031)}

\bibitem[Mag02]{Magyar_ergodic}
\bysame, \emph{Diophantine equations and ergodic theorems}, Amer. J. Math.
  \textbf{124} (2002), no.~5, 921--953. \MR{1925339 (2003f:37015)}

\bibitem[Mag07]{Magyar_discrepancy}
\bysame, \emph{On the distribution of lattice points on spheres and level
  surfaces of polynomials}, J. Number Theory \textbf{122} (2007), no.~1,
  69--83. \MR{2287111 (2007m:11104)}

\bibitem[Mag08]{Magyar_distance}
{\'A}kos Magyar, \emph{On distance sets of large sets of integer points},
  Israel J. Math. \textbf{164} (2008), 251--263. \MR{2391148 (2009d:11013)}

\bibitem[Mon94]{Montgomery}
H.~L. Montgomery, \emph{Ten lectures on the interface between analytic number
  theory and harmonic analysis}, CBMS Regional Conference Series in
  Mathematics, vol.~84, Published for the Conference Board of the Mathematical
  Sciences, Washington, DC, 1994. \MR{1297543 (96i:11002)}

\bibitem[MSW02]{MSW}
A.~Magyar, E.~M. Stein, and S.~Wainger, \emph{Discrete analogues in harmonic
  analysis: spherical averages}, Ann. of Math. (2) \textbf{155} (2002), no.~1,
  189--208. \MR{1888798 (2003f:42028)}

\bibitem[SS03]{SteinComplexAnalysis}
E.~M. Stein and R.~Shakarchi, \emph{Complex analysis}, Princeton Lectures in
  Analysis, II, Princeton University Press, Princeton, NJ, 2003. \MR{1976398
  (2004d:30002)}

\bibitem[Ste76]{Steinsphericalmaximalfunction}
E.~M. Stein, \emph{Maximal functions. {I}. {S}pherical means}, Proc. Nat. Acad.
  Sci. U.S.A. \textbf{73} (1976), no.~7, 2174--2175. \MR{0420116 (54 \#8133a)}

\bibitem[Ste93]{SteinHA}
\bysame, \emph{Harmonic analysis: real-variable methods, orthogonality, and
  oscillatory integrals}, Princeton Mathematical Series, vol.~43, Princeton
  University Press, Princeton, NJ, 1993, With the assistance of Timothy S.
  Murphy, Monographs in Harmonic Analysis, III. \MR{1232192 (95c:42002)}

\bibitem[Vau97]{Vaughan}
R.~C. Vaughan, \emph{The {H}ardy-{L}ittlewood method}, second ed., Cambridge
  Tracts in Mathematics, vol. 125, Cambridge University Press, Cambridge, 1997.
  \MR{1435742 (98a:11133)}

\bibitem[Woo12]{Wooley_efficient_congruencing1}
T.~D. Wooley, \emph{Vinogradov's mean value theorem via efficient
  congruencing}, Annals of Math. \textbf{175} (2012), 1575--1627.

\end{thebibliography}
\end{document}